\documentclass[12pt, a4paper, reqno]{amsart}
\usepackage{amscd, amssymb}
\usepackage{amsthm}
\usepackage{tikz}
\usetikzlibrary{arrows,positioning} 
\usepackage{mathrsfs}
\usepackage{bbm}
\usepackage{ stmaryrd }

\usepackage{hyperref}
\hypersetup{
	colorlinks=true,
	linkcolor=blue,
	filecolor=magenta,  
	urlcolor=cyan,
}

\usepackage{hyperref}
\usepackage[capitalise]{cleveref}

\def\Aut{\operatorname{Aut}}
\def\End{\operatorname{End}}

\def\dim{\operatorname{dim}}

\def\id{\operatorname{id}}

\def\Ber{\operatorname{Ber}}
\def\Ind{\operatorname{Ind}}

\def\Ad{\operatorname{Ad}}

\def\Hom{\operatorname{Hom}}

\def\Rep{\operatorname{Rep}}

\def\C{\mathbb{C}}

\def\Z{\mathbb{Z}}
\def\G{\mathbb{G}}

\def\F{\mathbb{F}}

\def\A{\mathbb{A}}

\def\OO{\mathcal{O}}

\def\UU{\mathcal{U}}

\def\EE{\mathcal{E}}

\def\DD{\mathcal{D}}
\def\NN{\mathcal{N}}

\def\CC{\mathcal{C}}
\def\II{\mathcal{I}}

\def\c{\mathfrak{c}}
\def\m{\mathfrak{m}}

\def\p{\mathfrak{p}}
\def\q{\mathfrak{q}}
\def\g{\mathfrak{g}}
\def\h{\mathfrak{h}}

\def\ft{\mathfrak{t}}
\def\k{\mathfrak{k}}
\def\l{\mathfrak{l}}
\def\s{\mathfrak{s}}
\def\o{\mathfrak{o}}

\def\d{\partial}

\def\ol{\overline}
\def\mod{\text{ mod}}

\def\Spec{\operatorname{Spec}}

\def\ul{\underline}

\def\Mod{\text{Mod}}
\def\sub{\subseteq}

\def\sVec{\text{sVec}}

\newtheorem{thm}{Theorem}[section]
\newtheorem{cor}[thm]{Corollary}
\newtheorem{lemma}[thm]{Lemma}
\newtheorem{conj}[thm]{Conjecture}
\newtheorem{prop}[thm]{Proposition}

\theoremstyle{definition}
\newtheorem{definition}[thm]{Definition}
\theoremstyle{remark}
\newtheorem{remark}[thm]{Remark}
\newtheorem{example}[thm]{Example}

\numberwithin{equation}{section}

\usepackage{relsize}
\usepackage{fancyhdr}
\usepackage[all,cmtip]{xy}

\begin{document}

\title{On symmetries of the Duflo-Serganova functor}

\author{Alexander Sherman}

\begin{abstract} We initiate a study of the symmetries of the Duflo-Serganova functor, $DS$.  In particular we give new constructions of Lie superalgebras, Lie supergroups, and associative superalgebras which act on this functor.  The main result is a realization that in the Kac-Moody setting, we find new odd infinitesimal symmetries of $DS$.  In addition, we connect our work to a computation of Heidersdorf and Weissauer which computed $DS_x$ for a maximal rank $x$ on Kac-modules for $GL(n|n)$, and extend the ideas and results to $P(n)$.  
\end{abstract}

\maketitle
\pagestyle{plain}

\section{Introduction}

In the study of the representation theory of Lie superalgebras, a large role has been played by the Duflo-Serganova functor, originally introduced in \cite{duflo2005associated} (see \cite{ds_bigpaper} for a recent survey of this functor and its applications). Given a Lie superalgebra $\g$ and an odd element $u\in\g_{\ol{1}}$, the condition that $[u,u]=0$ is nontrivial, and implies that $u$ defines a square zero operator on every representation of $\g$.  For a $\g$-module $M$, we write $M_u$ for the cohomology of this operator.  The functor $DS_u:M\mapsto DS_uM=M_u$ is called the Duflo-Serganova functor.   

More generally, if $u\in\g_{\ol{1}}$ is such that $c=[u,u]$ acts semisimply on a $\g$-module $M$, $u$ will define a square-zero operator on $M^c$ (the $c$-invariants on $M$), and we again denote the cohomology of $u$ on $M^c$ by $M_u=DS_uM$.  

A priori $M_u$ is simply a super vector space, but one of the first observations  made in \cite{duflo2005associated} is that under the adjoint action (assuming $\operatorname{ad}c$ is semisimple), $\g_u$ is a Lie superalgebra, and $M_u$ will have the natural structure of a $\g_u$-module.  In this way, $DS_u$ can be viewed as a functor from the category of $\g$-modules to the category of $\g_u$-modules.

The Duflo-Serganova functor has the power to reduce questions about modules over a given Lie superalgebra $\g$ to that over a smaller, often less complex Lie superalgebra $\g_u$.  It preserves the superdimension of a module, and has been used to compute superdimensions of simple modules for finite-dimensional Kac-Moody Lie superalgebras (see for instance \cite{HW}, \cite{entovaaizenbud2019dufloserganova}, and \cite{serganova_superdim}).  It defines a correspondence of central characters, and thus of blocks, between $\g$ and $\g_u$ (see Sec. 6 of \cite{ds_bigpaper}).  One can use $DS_u$ to define a homomorphism between (reduced) Grothendieck rings for $\g$ and $\g_u$ respectively, which is important in the study of the $\s\l(\infty)$-module structure on category $\OO$ for $\g\l(m|n)$ (see \cite{HPS}).  The functor $DS_u$ is even used to construct abelian envelopes of Deligne categories (see \cite{EHS}).  For more on applications of this functor, we refer to \cite{ds_bigpaper}.

Despite the intensive study of the Duflo-Serganova functor, certain natural, additional symmetries on $DS_u$ have been overlooked.  As we will demonstrate, one can construct larger (than $\g_u$) Lie superalgebras (and supergoups) which naturally act on $M_u$ for a $\g$-module $M$. We will see in examples that this larger Lie superalgebra tends to contain $\g_u$ as a subalgebra, and often takes the form $\g_u\times \C^{0|r}$, where $\C^{0|r}$ is an odd abelian Lie superalgebra of rank $r$.  Here $r$ is a statistic of $u$ that could be called its `rank'.  In certain cases we prove that this larger superalgebra acts faithfully on $M_u$ for certain modules $M$, meaning that these extra (infinitesimal) symmetries are non-trivial and of interest to study.

We now explain our two main constructions of additional symmetries, and give the major results obtained for each.  Afterward we will discuss (potential) applications and connections of these ideas.

\subsection{Construction of new Lie superalgebras}\label{section_intro_constr_lie_superalgeb} Let $u\in\g_{\ol{1}}$ be such that $c:=[u,u]$ acts semisimply on $\g$ under the adjoint representation, and let $\CC$ be a category of $\g$-modules on which $c$ acts semisimply on every module.  Write $\g^c$ for the centralizer of $c$ in $\g$, and let $\k\sub[u,\g^c]$ be a Lie superalgebra such that the restriction of any module in $\CC$ to $\k$ is semisimple over $\k$.  Then for any module $M$ in $\CC$ we have a natural isomorphism $DS_uM\cong DS_u(M^{\k})$ (Lemma \ref{invariance_ds_commute}).  Thus the Lie superalgebra $\g(u,\k):=(\g^\k/Z(\k))_u$ acts on $DS_u$, where $Z(\k)$ denotes the center of $\k$.  This Lie superalgebra always contains $\g_u$ as a subalgebra, and is often larger, as the following theorem demonstrates:

\begin{thm}\label{thm_intro_lie_super}
	Suppose that $\g$ is a finite-dimensional, symmetrizable Kac-Moody Lie superalgebra, and suppose that $u\in\g_{\ol{1}}$ has that $[u,u]$ is semisimple in $\g_{\ol{0}}$, and that the rank of $u$ is $r$ (see Lemma \ref{lemma_form} for definition of rank).  Then the largest Lie superalgebra of the form $\g(u,\k)$ which can be constructed is isomorphic to:
	\[
	\g_u\times\C^{0|r}.
	\]
\end{thm}
	Similar results are also proven for $\q(n)$ and $\p(n)$.  The main consequence of \cref{thm_intro_lie_super} is that we obtain an action of $\C^{0|r}$ on $M_u$ that commutes with the action of $\g_u$.
	
\subsection{Construction of a new supergroup}\label{sec_intro_supergroup} The construction given in \cref{section_intro_constr_lie_superalgeb} is very algebraic, and somewhat ad hoc; it requires making a choice of $\k$, for instance, and not all $\k$ give isomorphic Lie superalgebras $\g(u,\k)$.

We now present a more geometric construction of new symmetries, which is very natural.  Let $G$ be an algebraic supergroup; for an element $u\in\g_{\ol{1}}$ we may consider its adjoint vector field $u_{ad}$ on the Hopf superalgebra $\C[G]$, the coordinate ring of $G$.  If $c=[u,u]$ is such that its adjoint vector field acts semisimply on $\C[G]$, then $DS_{u_{ad}}\C[G]$ will again be a supercommutative Hopf superalgebra, and thus is given by the functions on some algebraic supergroup, which we call $\widetilde{G_u}$.  By functoriality, the supergroup $\widetilde{G_u}$ will naturally act on $DS_u$.  We have the following computations of $\widetilde{G_u}$ in certain cases:
\begin{thm}\label{thm_intro_ds_supergroup}
	\begin{enumerate}
		\item Let $u\in\g\l(m|n)_{\ol{1}}$ be such that $[u,u]=0$ and is of rank $r$; then
		\[
		\widetilde{GL(m|n)_u}\cong GL(m-r|n-r)\times\G_{a}^{0|r}.
		\]
		\item Let $u\in\q(n)_{\ol{1}}$ be such that $[u,u]=0$ and of rank $r$; then
		\[
		\widetilde{Q(n)_u}\cong Q(n-2r)\times\G_a^{0|r}.
		\]
	\end{enumerate}
\end{thm}
In the above $\G_{a}^{0|r}$ is the odd abelian Lie supergroup of rank $r$.  For a few more cases which we consider, see Theorem \ref{main_thm_supergroups}.  Note also that in \cite{localization} both $\widetilde{GL(m|n)_u}$ and $\widetilde{Q(n)_u}$ have been computed for any $u$ with $[u,u]$ semisimple in $\g_{\ol{0}}$, using \cref{thm_intro_ds_supergroup}; the general answer takes the same form as above.  

\subsection{Relation to \cite{HW}}\label{sec_intro_kac_mod} The first example (known to this author) of an approach to construct new symmetries was given in Section 26 of \cite{HW}, where for a certain maximal rank $u\in\g\l(n|n)_{\ol{1}}$ with $[u,u]=0$, they showed that there is a natural action of a Grassmann algebra on $n$ generators on $DS_u$; this is in spite of $\g_u$ being trivial for this choice of $u$.  They computed how this exterior algebra acts on $DS_u$ applied to Kac modules, showing in fact that certain Kac modules (the `maximally atypically' ones) become isomorphic to the regular representation of the exterior algebra under the Duflo-Serganova functor $DS_u$. 

We connect this computation of \cite{HW} to our approach, showing that one can view their work as computing the action of the enveloping algebra of a Lie supergroup that arises naturally from the Duflo-Serganova functor (see Section 6).  Further, the technique of Heidersdorf and Weissauer may be extended to the periplectic Lie superalgebra with its so-called thin and thick Kac-modules. In this case we obtain new computations of the action of the Duflo-Serganova functor for certain maximal rank elements in $\p(n)$ (see Theorem \ref{DS computation p case}).  See \cite{entovaaizenbud2019dufloserganova} for the computation of the rank one DS functor on simple modules for $\p(n)$. 

\subsection{Future directions: Kac modules}  We point toward two directions in the applications of our ideas.  The first relates to \cref{sec_intro_kac_mod}, and it is the question of understanding how Kac modules for type I superalgebras behave under the application of the Duflo-Serganova functor.  So far nothing is known in general, except the following result: for $u\neq0$ with $[u,u]=0$, if $L$ is any simple $\g_u$-module, and $K$ is a Kac-module for $\g$, then $[DS_uK:L]=[DS_uK:\Pi L]$ (see Section 8.5 of \cite{ds_bigpaper}).   An equivalent formulation is that $ds_u[K]=0$ in the reduced Grothendieck ring of $\g$-modules (see Section 2.3 and Section 8 of \cite{ds_bigpaper}). 

The fact that every simple $\g_u$-module $L$ has trivial supermultiplicity (i.e. $[DS_uK:L]-[DS_uK:\Pi L]$) in $DS_uK$  suggests the possibility of an odd supergroup that is acting freely on $DS_uK$ and that is commuting with the action of $\g_{u}$.  Indeed in this case, $\widetilde{G_u}$ will have Lie superalgebra $\g_u\times\C^{0|r}$, where $r$ is the rank of $u$, so we have such an odd supergroup present.  Write $\mathbb{O}\sub\widetilde{G_u}$ for the purely odd supergroup whose Lie superalgebra is given by $\C^{0|r}$.  

\begin{conj}\label{conj_intro_kac}
	If $K$ is a Kac-module for $\g$, then $\mathbb{O}$ acts freely on $DS_uK$.
\end{conj}

\cref{conj_intro_kac} would impose signficant structure on the $\g_u$-module structure of $DS_uK$, and thus would help in both the computations of its composition factors as well as its finer structure, e.g. Loewy layers.  Note that \cref{thm_kac_mod} is a special case of \cref{conj_intro_kac}.

\subsection{Connections to supergeometry}  The second connection to our work is inspired by the recent work \cite{localization}.  There, a localization theorem is proven that computes, under some conditions, $DS_u\C[X]$, where $X$ is a smooth affine algebraic supervariety, $\C[X]$ is its superalgebra of functions, and $u$ is a vector field on $X$ such that $[u,u]$ acts semisimply on $\C[X]$.  A special case is given in \cref{sec_intro_supergroup} when $X=G$ is an algebraic supergroup, and $u$ is an adjoint vector field on $G$.

Now suppose that $G$ is a supergroup and $X$ is a smooth, affine $G$-supervariety; for instance, $X$ may be a homogeneous $G$-space, such as a supersymmetric space.  Then $\widetilde{X_u}:=\Spec\C[X]_u$ will naturally have the structure of a $\widetilde{G_u}$-supervariety.  Further, a $G$-invariant differential operator $D$ on $X$ will naturally descend to a $\widetilde{G_u}$-invariant differential operator on $\widetilde{X_u}$, giving rise to a morphism of algebras 
\[
\DD^{G}(X)\to\DD^{\widetilde{G_u}}(\widetilde{X_u}),
\]
where for an $H$-supervariety $Y$ we write $\DD^{H}(Y)$ for the algebra of $H$-invariant differential operators on $Y$.   Such algebras of invariant differential operators are very important in the theory of supersymmetric spaces, with one special case being that of $\DD^{G\times G}(G)$ which gives $Z(\UU\g)$, the center of the enveloping superalgebra of $\g=\operatorname{Lie}G$. In this case, it was observed already in \cite{duflo2005associated} that there is a natural map $Z(\UU\g)\to Z(\UU\g_{u})$, and its image is well-understood and important in the study of blocks of the category of $\g$-modules.  No analogous maps have been known for general supersymmetric spaces; however in \cite{localization}, it is found that for many cases when $X$ is a supersymmetric space, $\widetilde{X_u}$ will again be a supersymmetric space for $\widetilde{G_u}$, and thus we obtain such a natural homomorphism $\DD^{G}(X)\to\DD^{\widetilde{G_u}}(\widetilde{X_u})$.  Relating this homomorphism to the Harish-Chandra maps in each case would be of great interest.

Further, in many cases we find that our additional odd generators of $\widetilde{G_u}$ act nontrivially on $\widetilde{X_u}$.  One example is when $X=GL(m|2n)/OSp(m|2n)$, with $G=GL(m|2n)$.  Here it is shown in \cite{localization} that if $u\in\o\s\p(m|2n)_{\ol{1}}$ is of rank 1, we have \linebreak $\widetilde{G_u}=GL(m-2|2(n-1))\times\mathbb{G}_{a}^{0|2}$, and 
\[
\widetilde{X_u}=\left(GL(m-2|2(n-1))/OSP(m-2|2(n-1))\right)\times\left(\mathbb{G}_{a}^{0|2}/\mathbb{G}_a^{0|1}\right).
\]
Here $\mathbb{G}_a^{0|r}$ denotes the purely odd, connected supergroup of rank $r$, and $\mathbb{G}_{a}^{0|1}\sub\mathbb{G}_a^{0|2}$ is embedded diagonally.  It would be of interest to understand either geometrically or representation theoretically what is the origin of the extra odd infinitesimal action in this case.

\subsection{Outline of paper} Section 2 will begin with preliminaries, and in particular we give a precise definition of symmetries of the Duflo-Serganova functor.  We will not use this definition in later sections.  We give an example for the general linear supergroup $GL(1|1)$.

We will introduce, separately, three different approaches to constructing symmetries of the Duflo-Serganova functor.  The first construction gives $\g(u,\k)$ as described in \cref{section_intro_constr_lie_superalgeb}, and Section 3 is devoted to the study of these Lie superalgebra.  In Section 4 we introduce the construction given in \cref{sec_intro_supergroup}, which is geometric and gives $\widetilde{G_u}$; we make a few general statements, but most of the section is devoted to certain important computations of this supergroup.  Finally, the third approach is discussed in Section 5; namely we produce an associative superalgebra which is a subquotient of $\UU\g$ and acts on $DS_u$.  As of yet we do not understand how to compute this superalgebra in most cases, but we introduce it because it generalizes, in some sense, the construction given in \cite{HW}.  We explain this connection, and further generalize it to the case of $\p(n)$; in the process we compute $DS_u$ on thin and thick Kac-modules when $u$ is a certain maximal rank element.

Finally, in section 6, we will connect the second and third approaches in the special case when the ideas of Heidersdorf and Weissauer apply, that is when $\g$ has a compatible $\Z$-grading $\g=\g_{-1}\oplus\g_0\oplus\g_1$, including the cases of $\g\l(n|n)$ and $\p(n)$.  Here, if $u$ is maximal rank and satisfies certain other properties, we may explicitly show that in some sense the third type of construction gives the enveloping superalgebra of the Lie supergroup given in the second, more geometric construction.

\subsection{Acknowledgments}  The author thanks Inna Entova-Aizenbud, Maria Gorelik, Vladimir Hinich, and Vera Serganova for many helpful discussions.  The author further thanks an anonymous referee for several helpful suggestions.  This research was partially supported by ISF grant 711/18 and NSF-BSF grant 2019694.

\section{Preliminaries}

\subsection{Notation} We work throughout over $\C$.  For a super vector space $V$ we write $V=V_{\ol{0}}\oplus V_{\ol{1}}$ for its parity decomposition.  Write $\sVec$ for the category of super vector spaces with even endomorphisms.

If $V=V_{\ol{1}}$ is a purely odd super vector space, we will write $SV_{\ol{1}}$ for the supersymmetric algebra on $V$, meaning that as a vector space it is really the exterior algebra on $V_{\ol{1}}$.

Let $\g$ be a Lie superalgebra, and let $\CC$ be a full subcategory of the category of $\g$-modules which contains the adjoint representation,  some faithful $\g$-module, and is closed under tensor product.  Let $c\in\g_{\ol{0}}$ be such that $c$ acts semisimply on every module in $\CC$; for instance the element $c=0$ always satisfies this condition, and this is the prototypical example.  For a module $M$, we write $M^c$ for the $c$-invariants on $M$; by our assumption, this will define an exact functor from $\CC$ to $\sVec$.  

Now for $u\in\g_{\ol{1}}$ such that $2u^2:=[u,u]=c$, $u$ defines a square-zero endomorphism on $M^c$ for any $M$ in $\CC$, and we write $M_u$ for its cohomology.  
\begin{definition}
	Define the Duflo-Serganova functor $DS_u:\CC\to\sVec$ by $DS_uM:=M_u$.
\end{definition}  

We note the following properties of $DS_u$, which were originally proven in Sec. 2 of \cite{ds_bigpaper}:
\begin{enumerate}
	\item $DS_u$ is a rigid tensor functor, i.e. we have canonical isomorphisms $(M\otimes N)_u\cong M_u\otimes N_u$ and $DS_uM^*\cong (DS_uM)^*$.
	\item $\C_u=\C$, and $\g_u$ is a Lie superalgebra such that $M_u$ naturally has the structure of $\g_u$-module.
	\item $DS_u$ is middle exact; in fact it takes a short exact sequence $0\to M'\to M\to M''\to 0$ to a long exact sequence
	\[
	\cdots\to \Pi M_u''\to M_u'\to M_u\to M_u''\to\Pi M_u'\to \cdots
	\]

\end{enumerate}

\begin{remark}
	Note that often $DS_u$ is thought of as a functor $\g$--$\mod\to\g_u$--$\mod$; however we choose instead to view it as simply a functor to $\sVec$ in order to simplify the study of its symmetries, as a functor.
\end{remark}

\subsection{Symmetries}  For every supercommutative $k$-algebra $R$ we we have functor $DS_u^R:\CC\otimes_k R\to R-\Mod$ defined in the natural way.  We write $\ul{\Aut}^{\otimes}(DS_u^R)$ for the group of automorphisms of $DS_u^R$, which consists of all $R$-linear invertible natural transformations $\alpha$ of $DS_u^R$ such that:
\begin{enumerate}
	\item For every morphism $f:M\to N$ in $\CC$ the following diagram is commutative:
	\[
	\xymatrix{
M_u\otimes_kR\ar[d]^{f_u\otimes \id_R} \ar[r]^{\alpha_M}	& M_u\otimes_K R \ar[d]^{f_u\otimes\id_R} \\ N_u\otimes_kR \ar[r]^{\alpha_N} & N_u\otimes_K R 
} 
	\]
	\item If $M,N$ are in $\CC$, then the following diagram commutes:
	\[
	\xymatrix{
 	(M\otimes N)_u\otimes_k R \ar[d] \ar[rr]^{\alpha_{M\otimes N}}	&& (M\otimes N)_u\otimes_k R \ar[d]\\
 	(M_u\otimes_kR)\otimes_R(N_u\otimes_kR) \ar[rr]^{\alpha_M\otimes\alpha_N} && (M_u\otimes_kR)\otimes_R(N_u\otimes_kR).
}	\]
	\item $\alpha_{\C}=\id_R$.
\end{enumerate} 
In this way may define a functor $\ul{\Aut}^{\otimes}(DS_u)$ from the category of supercommutative superalgebras to groups.  We say that a supergroup $L$ acts on $DS_u$ if we have a natural transformation of functors $L\to\ul{\Aut}^{\otimes}(DS_u)$.

The following lemma is straightforward from the definition.
\begin{lemma}
	Suppose that $L$ is a Lie supergroup such that for every module $M$ in $\CC$ we have an action of $L$ on $M_u$ which is natural in $M$.  Then $L$ acts on $DS_u$.
\end{lemma}

Similarly we define $\ul{\mathfrak{aut}}^{\otimes}(DS_u^R)$ as the subspace of $R$-linear natural endomorphisms $a$ of $DS_u^R$ such that
\begin{enumerate}
	\item For every morphism $f:M\to N$ in $\CC$ the following diagram is commutative:
	\[
	\xymatrix{
		M_u\otimes_kR\ar[d]^{f_u\otimes 1} \ar[r]^{a_M}	& M_u\otimes_K R \ar[d]^{f_u\otimes 1} \\ N_u\otimes_kR \ar[r]^{a_N} & N_u\otimes_K R 
	} 
	\]
	\item If $M,N$ are in $\CC$, then the following diagram commutes:
	\[
	\xymatrix{
		(M\otimes N)_u\otimes_k R \ar[d] \ar[rrr]^{a_{M\otimes N}} &&& (M\otimes N)_u\otimes_k R \ar[d]\\
		(M_u\otimes_kR)\otimes_R(N_u\otimes_kR) \ar[rrr]^{a_M\otimes\id+\id\otimes a_N} &&& (M_u\otimes_kR)\otimes_R(N_u\otimes_kR).
	}	\]
	\item $a_{\C}=0$.
\end{enumerate} 

Thus we have defined a functor $\ul{\mathfrak{aut}}^{\otimes}(DS_u)$ from the category of supercommutative superalgebras to Lie algebras.  We say that a Lie superalgebra $\l$ acts on $DS_u$ if we have a natural transformation of functors $\l\to\ul{\mathfrak{aut}}^{\otimes}(DS_u)$.

Again the following lemma is a straightforward check.
\begin{lemma}
	Suppose that $\l$ is a Lie superalgebra such that for every module $M$ in $\CC$ we have an action of $\l$ on $M_u$ which is natural in $M$.  Then $\l$ acts on $DS_u$.
\end{lemma}

\subsection{An example}  

\begin{example}\label{gl(1|1) example}
	Let $\g=\g\l(1|1)$ and let $\CC=\Rep GL(1|1)$, i.e. the finite-dimensional representations of the general linear supergroup $GL(1|1)$.  Present $\g\l(1|1)$ as the matrices
	\[
	\begin{bmatrix}
	I+h & E\\ F & I-h
	\end{bmatrix}
	\]
	We compute $DS_E$ on $\CC$.  By the known classification of indecomposable modules in $\CC$, the only indecomposables on which $DS_E$ acts nontrivially are the parity shifts of the Berezinian modules $(\Pi)\Ber^n$, for $n\in\Z$, along with the parity shifts of the Kac modules $(\Pi)K^+(0)\otimes\Ber^n$; on these modules $E$ acts trivially, while $F$ acts nontrivially on the Kac-modules.  Further $h$ normalizes $E$; thus we obtain a natural action of the Lie superalgebra $\langle h,F\rangle$ on $DS_E$.  A study of the morphisms between these modules shows that $\langle h,F\rangle=\ul{\mathfrak{aut}}^{\otimes}(DS_E)$ in this case.
\end{example}

\section{First approach: natural Lie superalgebras which act by symmetries}

Let $\k\sub[u,\g^c]$ be a subalgebra which acts semisimply on all objects of $\CC$. In particular we have $[u,\k]=0$. 
\begin{lemma}\label{invariance_ds_commute}
	If $V$ is in $\CC$, then the inclusion $V^{\k}\to V$ induces a natural isomorphism $(V^{\k})_u\cong V_u$.
\end{lemma}  
\begin{proof}
	We may decompose $V$ as a $\k$-module as $V=V^{\k}\oplus V'$, and $u$ must preserve this decomposition since $[u,\k]=0$.  Further $\k$ acts on $(V')_u$, which will be a semisimple $\k$-module with no trivial constituents.  But $\k$ acts trivially on $V_u$ since $\k\sub[u,\g^c]$, so necessarily $(V')_u=0$.
\end{proof}
%

\begin{cor}
	We have a natural isomorphism of functors $DS_u\cong DS_u\circ (-)^{\k}$.
\end{cor}

\subsection{Construction of $\g(u,\k)$} We have a natural action of $\g^{\k}=\c(\k)$ on $V^{\k}$.  This descends to an action of $\c(\k)/Z(\k)$ on $V^{\k}$. Taking cohomology, we obtain an action of the Lie superalgebra 
\[
\g(u,\k):=(\c(\k)/Z(\k))_u
\]
on $(V^{\k})_u$.  Thus we obtain a new Lie superalgebra that acts on $DS_u$.

\begin{example}
	We return to Example \ref{gl(1|1) example}, and in particular retain the notation from that example.  Let $u\in\g_{\ol{1}}$ be an arbitrary non-zero element; then $u^2=c$ is semisimple and is a multiple of the central element $I$ of $\g$, thus $\g=\g^c$.  We observe that in any case, $\k:=\C\langle I\rangle\sub [u,\g]$, and $c(\k)/\k\cong\p\g\l(1|1)$.  Therefore $\g(u,\k)=\p\g\l(1|1)_u\cong\C^{0|1}$ naturally acts on the Duflo-Serganova functor in this case.   If $u^2=0$ in this case, then we in fact obtain a non-trivial action as we saw in Example \ref{gl(1|1) example}.
\end{example}

\subsection{Maximality of $\g(u,\ft)$} Choose a maximal toral subalgebra $\ft$ of $\k$, that is a maximal even subalgebra such that every element in it acts semisimply on all modules in $\CC$.  Then $\mathfrak{t}$ must be abelian, and in particular must contain $Z(\k)$.  We obtain an inclusion $\c(\k)\sub\c(\ft)$, and since $Z(\k)\sub Z(\ft)=\ft$, we obtain a map $\phi:\c(\k)/Z(\k)\to\c(\ft)/\ft$, which is injective since $\c(\k)\cap\ft=Z(\k)$.  Taking cohomology, we obtain a map of Lie superalgebras:
\[
\phi_u:\g(u,\k)\to\g(u,\ft).
\]
Further we have a commutative diagram
\[
\xymatrix{\g(u,\k)\ar[r]\ar[d]^{\phi_u} & \End((V^{\k})_u)\ar[d]\\ \g(u,\ft)\ar[r] & \End((V^{\ft})_u)}.
\]
\begin{lemma}
	$\phi_u$ is injective.
\end{lemma}
\begin{proof}
	Suppose that $y\in\c(\k)^c$ such that $[y,u]\in Z(\k)$, and there exists $z\in\c(\ft)^c$ such that $[u,z]=y+h$ for some $h\in\ft$.  Since $\k$ acts semisimply, we may decompose $z=z_0+z'$ and $h=h_0+h'$, where $z_0\in\c(\k)^c$, $h_0\in Z(\k)$, and $z',h'$ lie in the $\k$-complement in $\g^c$.  Then we obtain the formula
	\[
	y+h_0-[u,z_0]=[u,z']-h'.
	\]
	However the LHS lies in $\c(\k)^c$ while the RHS lies in the $\k$-complement to $\c(\k)^c$ in $\g^c$.  Thus both must be zero, and we find that
	\[
	y+h_0=[u,z_0].
	\]
	However this exactly means $y$ is 0 in $(\c(\k)/Z(\k))_u$.  Therefore $\phi_u$ is injective.
	
\end{proof}

Write $K\sub GL(\g)$ for the subgroup of $GL(\g)$ integrating $\k$.  Let $W$ denote the Weyl group obtained from the action of $N_K(\ft)$ on $\ft$.  Then $W$ naturally acts on $\g(u,\ft)$.  The following lemma is follows from that fact that the action of $W$ comes from $K$, which acts trivially on $DS_u$.

\begin{lemma}\label{W action}
	The image of $\phi_u:\g(u,\k)\to\g(u,\ft)$ lands in $\g(u,\ft)^W$.	Further, the action of $\g(u,\ft)$ on $DS_u$ factors through the coinvariants of $W$.
\end{lemma}

The above lemma shows that the Lie superalgebra $\g(u,\ft)$ does not need to act faithfully.

\begin{lemma}\label{quotient_ds_is_larger}
	With the same assumptions on $\k$ as above, the natural map
	\[
	\c(\k)_u\to\g(u,\k)
	\]
	is injective.  In particular, since $\g_{u}\cong\c(\k)_u$, $\g_u$ is a subalgebra of $\g(u,\k)$.
\end{lemma}
\begin{proof}
	Suppose $y\in\c(\k)^c$ such that $[u,y]=0$ and there exists $w\in\c(\k)^c$ such that $[u,w]=y+z$ where $z\in Z(\k)$.  But $Z(\k)\sub\k\sub[u,\g^c]$, so $z\in[u,\g^c]$, and we may write $z=[u,v]$ for some $v\in\g^c$.  
	
	However if we decompose $v=v_0+v'$ where $v_0\in\c(\k)^c$ and $v'$ is in the $\k$-complement of $\c(\k)^c$ in $\g^c$, then $z=[u,v_0]+[u,v']$, and since $u\in\c(\k)$ this implies that $z=[u,v_0]$.  Therefore $y=[u,u-v_0]$, so we are done.
\end{proof}

\begin{lemma}
	Suppose that $\ft'\sub\ft\sub[u,\g^c]$ are toral subalgebras of $[u,\g^c]$.  Consider the natural morphisms of algebras
	\[
	\xymatrix{C(\ft)/\ft'\ar[r]^\phi\ar[d]_\psi & C(\ft')/\ft'\\ C(\ft)/\ft &}.
	\]
	This induces morphisms:
	\[
	\xymatrix{\left(C(\ft)/\ft'\right)_u\ar[r]^{\phi_u}\ar[d]_{\psi_u} & \g(u,\ft')\\ \g(u,\ft) &}.
	\]
	Then we have $\phi_u$ is an isomorphism and $\psi_u$ is injective.
\end{lemma}
\begin{proof}
	The morphism $\phi$ is an inclusion of $\c(\ft)/\ft'$-modules.  Observe that $\ft/\ft'\sub[u,(c(\ft)/\ft')^c]$ is a subalgebra with semisimple representation theory and $(\c(\ft')/\ft')^{\ft/\ft'}=\c(\ft)/\ft'$.  Thus by Lemma \ref{invariance_ds_commute} the morphism $\phi_u$ is an isomorphism.
	
	The map $\psi_u$ is injective by Lemma \ref{quotient_ds_is_larger} applied to $\g=\c(\ft)/\ft'$ and $\k=\ft/\ft'$.   
\end{proof}

In summary, we have shown the following proposition.
\begin{prop}
	 Given $u\in\g_{\ol{1}}$ with $[u,u]=c$ a semisimple operator, then the largest Lie superalgebra of the form $\g(u,\k)$ is obtained by taking $\k=\ft\sub[u,\g^c]$ a maximal toral subalgebra of $[u,\g^c]$.
\end{prop}  

\subsection{Classical Lie superalgebras}\label{sec_classical} Suppose that $\g$ is quasireductive, meaning that $\g_{\ol{0}}$ is reductive and $\g_{\ol{1}}$ is a finite-dimensional, semisimple $\g_{\ol{0}}$-module.  We set
\[
\g_{\ol{1}}^{hom}:=\{u\in\g_{\ol{1}}:[u,u]\text{ is semisimple in }\g_{\ol{0}}\}.
\]
Now we consider the case when $\g$ is finite-dimensional, symmetrizable Kac-Moody (see \cite{serganova_KM}).
\begin{lemma}\label{lemma_form}
	Let $\g$ be a finite-dimensional, symmetrizable Kac-Moody Lie superalgebra, and let $u\in\g_{\ol{1}}^{hom}$.  Then there exists a Cartan subalgebra $\ft\sub\g_{\ol{0}}$ and mutually orthogonal, linearly independent isotropic roots $\alpha_1,\dots,\alpha_r$ such that 
	\[
	u=u_{\alpha_1}+\dots+u_{\alpha_r}+c_1v_{\alpha_1}+\dots+c_kv_{\alpha_r}
	\]
	where $u_{\alpha_i}\in\g_{\alpha_i}$, $v_{\alpha_i}\in\g_{-\alpha_i}$, and $c_i\in\C$.   We say that an element $u$ of the above form has rank $r$.
\end{lemma}

\begin{proof}
	Write $c=u^2$, and consider the Lie subalgebra $\c(c)$; by Lem. 3.1 of \cite{splitting}, we can write $\c(c)=\g^1\times\cdots\times\g^k$, where each $\g^i$ is itself Kac-Moody or even abelian.  Since $u\in\c(c)$, we can write $u=u_1+\dots+u_k$ where $u_i\in\g^i$ such that $u_i^2$ is central in $\g^i$. 
	
	It therefore suffices to prove the statement in the case when $\g$ is indecomposable Kac-Moody and $u^2=c$ is central.  If $c=0$, the statement is proven in Thm 5.1 of \cite{ds_bigpaper}.  If $c\neq0$, then by the classification of finite-dimensional Kac-Moody Lie superalgebras, $\g=\g\l(n|n)$ for some $n$, and $c$ is a non-zero scalar multiple of the identity matrix; the statement in this case is straightforward.
\end{proof}

\begin{remark}
	Note that in the above, proof, $k$ will be at most $r$, and will be equal to $r$ if and only if $\c(\g)$ is a product of $k$ copies of $\g\l(1|1)$.  The above proof technique is very similar to what is used in Prop. 3.3 of \cite{splitting}.
\end{remark}

\begin{example}
	Let us do an example to see how the proof of \cref{lemma_form} works in a particular case.  Consider the following element $u$ of $\g\l(4|4)$:
	\[
	\begin{bmatrix} &&&& 1 &&& \\ &&&&&1&& \\ &&&&&&1& \\ &&&&&&&1 \\ 1 &&&&&&& \\ & 1 &&&&&& \\ &&-1&&&&&\\ &&&0&&&&\end{bmatrix}
	\]
	we have $c=u^2$ is given by $\operatorname{diag}(1,1,-1,0,1,1,-1,0)$, and thus its centralizer is given by $\c(c)=\g\l(2|2)\times\g\l(1|1)\times\g\l(1|1)$.  Under this decomposition, $u$ is given by
	\[
	\left(\begin{bmatrix} 0 & 0 & 1 & 0\\ 0 & 0 & 0 &1\\ 1 & 0 & 0 &0 \\ 0 & 1 & 0 & 0\end{bmatrix}, \begin{bmatrix}0 & 1\\-1&0	\end{bmatrix}, \begin{bmatrix}0 & 1\\ 0 & 0	\end{bmatrix}\right).
	\]
	From here the decomposition desired decomposition into a sum of root vectors with $r=4$ is clear.
\end{example}

\begin{thm}\label{thm construction 1}
	Let $\g$ be a finite-dimensional Kac-Moody Lie superalgebra. Let 
	\[
	u=u_{\alpha_1}+\dots+u_{\alpha_r}+c_1v_{\alpha_1}+\dots+c_kv_{\alpha_r}
	\]
	as in Lemma \ref{lemma_form}.  Then 
	\[
	\ft'=\sum\limits_i[\g_{\alpha_i},\g_{-\alpha_i}]\sub [u,\g^c]
	\]
	is a maximal toral subalgebra of $[u,\g^c]$, and we have:
	\[
	\g(u,\ft')\cong \g_u\times \C\langle v_{\alpha_1},\dots,v_{\alpha_r}\rangle\cong\g_u\times\C^{0|r}
	\]
	where $v_{\alpha_1},\dots,v_{\alpha_r}$ are nonzero root vectors of weight $-\alpha_1,\dots,-\alpha_r$.  
\end{thm}
\begin{proof}
	Indeed we have $\c(\ft')=\g_u\times\g\l(1|1)^r$, where the copies of $\g\l(1|1)$ correspond to root subalgebras of $\alpha_1,\dots,\alpha_r$.  The computation from here is straightforward.
\end{proof}

\subsubsection{$\q(n)$} Next consider $\g=\q(n)$; it has a matrix presentation as elements of $\g\l(n|n)$ of the form:
\[
T_{A,B}=\begin{bmatrix}
A & B\\ B & A
\end{bmatrix}
\]
where $A,B\in\g\l(n)$ are arbitrary.  When $A=0$ we will simply write $T_B:=T_{0,B}$, and when $B=0$ we will write $T^A=T_{A,0}$. The following lemma follows from the fact that the action of $GL(n)$ on $\q(n)_{\ol{1}}$ is the adjoint representation.
\begin{lemma}
	For $\g=\q(n)$, every $u\in\g_{\ol{1}}^{hom}$ is conjugate to an element of the form
	\[
	u=T_{E_{12}}+T_{E_{34}}+\dots+T_{E_{2r-1,2r}}+c_1T_{E_{2r+1,2r+1}}+\dots+c_kT_{E_{2r+k,2r+k}}
	\]
	where $c_i\in\C$.  In this case we say $u$ has rank $r+k/2$.
\end{lemma}
Then we easily have:
\begin{prop}\label{superalgebra q case}
	Let $u=T_{E_{12}}+T_{E_{34}}+\dots+T_{E_{2r-1,2r}}+c_1T_{E_{2r+1,2r+1}}+\dots+c_kT_{E_{2r+k,2r+k}}$ as in the above lemma, with $c_1,\dots,c_k\neq0$.  Let 
	\[
	\ft':=\operatorname{Span}\{T^{E_{11}+E_{22}},T^{E_{33}+E_{44}},\dots,T^{E_{2r-1,2r-1}+E_{2r,2r}},T^{E_{2r+1,2r+1}},\dots,T^{E_{2r+k,2r+k}}\}.
	\]
	Then $\ft'\sub[u,\q(n)]$ is a maximal toral subalgebra, and we have
	\begin{eqnarray*}
	\g(u,\ft')& = &\q(n-2r-k)\times\C\langle T_{E_{21}},T_{E_{43}},\dots,T_{E_{2r,2r-1}},T_{E_{2r+1,2r+1}},\dots,T_{E_{2r+k,2r+k}}\rangle\\
	         & = &\q(n-2r-k)\times\C^{0|r+k}.
	\end{eqnarray*}
\end{prop}

\subsubsection{$\p(n)$} Finally we look at the case of $\g=\p(n)$.  Recall that $\p(n)$ may be presented as the subalgebra of $\g\l(n|n)$ consisting of the matrices:
\[
\begin{bmatrix}
A & B\\ C& -A^{t}
\end{bmatrix}
\]
where $B^t=B$ and $C^t=-C$. Let $\h\sub\p(n)$ denote the diagonal matrices, which gives a Cartan subalgebra of $\p(n)$.  Then the odd roots with respect to $\h$ are $\{\epsilon_i+\epsilon_j\}_{1\leq i\leq j\leq n}\cup\{-\epsilon_{i}-\epsilon_j\}_{1\leq i<j\leq n}$.  We write $e_{\alpha}$ for a chosen nonzero root vector of weight $\alpha$.
\begin{lemma}
	Let $u\in\g_{\ol{1}}^{hom}$.  Then $u$ is conjugate to an element of the form
	\[
	u=e_{\epsilon_1+\epsilon_2}+\dots+e_{\epsilon_{2r-1}+\epsilon_{2r}}+de_{2\epsilon_{n}}+c_1e_{-\epsilon_1-\epsilon_2}+\dots+c_se_{-\epsilon_{2s-1}-\epsilon_{2s}}
	\]
where $c_i\in\C$, and either $d=0$, or $d=1$ in which case $2r,2s<n$.  
\end{lemma}

\begin{proof}
	The proof is an exercise in linear algebra; we begin with a matrix $u=\begin{bmatrix}0 & B\\ C& 0\end{bmatrix}$ such that $u^2$ is semisimple.  We may begin by conjugating $B$ to a matrix of the form $B=\begin{bmatrix} I_r & 0 \\ 0 & 0	\end{bmatrix}$.  The stabilizer of $B$ in $GL(n)$ is given by matrices of the form 
	\[
	\begin{bmatrix}
	X & Y \\0 & Z
	\end{bmatrix}
	\]
	where $X^t=X^{-1}\in GL(r)$, $Z\in GL(n-r)$, and $Y$ is an arbitrary $r\times(n-r)$-matrix.  Now consider $C=\begin{bmatrix} C_1 & C_2\\ -C_2^t & C_3\end{bmatrix}$, where $C_1^t=-C_1$, $C_3^t=-C_3$.  Then the condition that $u$ is semisimple implies first of all that $C_1$ is semisimple. Conjugating by $O(n)$, we may assume that $C_1$ looks the following way:
	\[
	C_1=\begin{bmatrix} J_1 & 0 & 0 & \dots& & 0\\ 0 & J_2 & 0 & & & \vdots\\ 0  & 0 & J_3 &  & &\\ \vdots & &&\ddots && \end{bmatrix}.
	\]	
	where $J_i=\begin{bmatrix} 0 & \lambda_i\\ -\lambda_i & 0 \end{bmatrix}$.  We can further assume that $J_1,\dots,J_i\neq0$, but $J_k=0$ for $k>i$. Thus $C_1$ is in fact of the form
	\[
	C_1=\begin{bmatrix}
	R & 0\\0 & 0
	\end{bmatrix}
	\]
	where $R$ is invertible and semisimple.  Now the condition that $u^2$ is semisimple implies that $C_2$ must have the form
	\[
	C_2=\begin{bmatrix} S & T\\0 & 0\end{bmatrix}.
	\]
	Now we may conjugate by a matrix of the form $\begin{bmatrix}I_r & Y\\0 & I_{n-r}\end{bmatrix}$ in order to assume that in fact $C_2=0$.  From here, the result is straightforward.
\end{proof}

\begin{prop}\label{superalgebra p case}
 	Let	$u\in\g_{\ol{1}}^{hom}$ be of the form
	\[
	u=e_{\epsilon_1+\epsilon_2}+\dots+e_{\epsilon_{2r-1}+\epsilon_{2r}}+de_{\epsilon_{n}}+c_1e_{-\epsilon_1-\epsilon_2}+\dots+c_se_{-\epsilon{2s-1}-\epsilon_{2s}}.
	\]
Let $t=\max(r,s)$.  Then
	\[
	\ft'=\sum\limits_{i=1}^{t}[\g_{\epsilon_{2i-1}+\epsilon_{2i}},\g_{-\epsilon_{2i-1}-\epsilon_{2i}}]\sub[u,\p(n)]
	\]
	is a maximal toral subalgebra of $[u,\p(n)]$.  We have
	\[
	\g(u,\ft')=\p(n-t-d)\times\C^{0|t}.
	\]
\end{prop}

\begin{remark}\label{non_faithful_rmk}
We remind that Lemma \ref{W action} shows that $\g(u,\k)$ need not act faithfully on $DS_u$, and thus neither do the Lie superalgebras constructed in Theorem \ref{thm construction 1} or Propositions \ref{superalgebra q case} and \ref{superalgebra p case}.  In some cases this construction is too simple to produce many new symmetries of $DS_u$, because it is entirely linear.  In the following sections we will describe constructions that are less linear in nature, and which produce symmetries that are significantly more complex to describe.  
\end{remark}

\section{Second approach: natural supergroups which act by symmetries}

\subsection{The supergroup $\widetilde{G_u}$} Let $G$ be an affine algebraic supergroup over $\C$, and write $\g=\operatorname{Lie}G$ for its Lie superalgebra.  Write $\C[G]$ for the the Hopf superalgebra of functions on $G$ (see Chpt. 11 of \cite{carmeli2011mathematical} for definitions), and let $\Delta$ the coproduct map on $\C[G]$.  Then $G$ acts on itself by conjugation, and inducing an action of $\g$ on $G$ by vector fields.  For $u_e\in T_eG$, the vector field we get is
\[
u_{ad}=u_R+u_L=(1\otimes u_e-u_e\otimes 1)\circ\Delta.
\] 
where $u_R=(1\otimes u_e)\circ\Delta$ and $u_L=-(u_e\otimes 1)\circ\Delta$ are the infinitesimal right and left translations in the $u_e$ direction.  We refer to $u_{ad}$ as the adjoint vector field of $u=u_e$.

This adjoint action of $\g$ on $\C[G]$ respects the Hopf superalgebra structure, i.e. all morphisms coming from the Hopf superalgebra structure are morphisms of $\g$-modules.  Let $u\in\g_{\ol{1}}$ be such that $[u,u]=c$ has a semisimple adjoint vector field.  Then since $DS_u$ is a tensor functor, $\C[G]_u=DS_{u_{ad}}\C[G]$ is a supercommutative Hopf superalgebra.  
\begin{definition}
	We write $\widetilde{G_u}$ for the algebraic supergroup with $\C[\widetilde{G_u}]=\C[G]_u$, and $\widetilde{\g_u}$ for its Lie superalgebra.
\end{definition}

If $V$ is a $G$-module, then we have morphism $V\to V\otimes\C[G]$, so when we take $DS_u$ we obtain a morphism $V_u\to V_u\otimes \C[G]_u=V_u\otimes \C[\widetilde{G_u}]$, which gives $V_u$ the natural structure of a $\widetilde{G_u}$-module; in this way $\widetilde{G_u}$ acts on the DS functor. 

\begin{remark}
	The supergroup $\widetilde{G_u}$ need not act faithfully on $DS_u$; an easy counterexample is as follows.  Let $G=\G_{a}^{0|1}$ denote the abelian, purely odd Lie supergroup, and let $u\in\g_{\ol{1}}$ be a non-zero element.  Then $u_{ad}=0$, so $\widetilde{G_u}=\G_{a}^{0|1}$; on the other hand $\ul{Aut}^{\otimes}(DS_u)$ is the trivial Lie supergroup.
\end{remark}

The following theorem lists the main results of this section.  For the definition of rank in the $GL(m|n)$ and $Q(n)$ cases, see Section \ref{sec_classical}.  We use the notation $\G_a^{0|n}$ for the connected Lie supergroup with Lie superalgebra $\C^{0|n}$.
\begin{thm}\label{main_thm_supergroups} 
	
	\begin{enumerate}
		\item If $G=GL(m|n)$ and $u\in\g_{\ol{1}}$ with $u^2=0$ of rank $s$, then we have
		\[
		\widetilde{G_u}=GL(m-s|n-s)\times\G_{a}^{0|r}.
		\]
		\item For $G=GL(1|1)$ with $u\in\g_{\ol{1}}^{hom}$ arbitrary and $u\neq0$, we have
		\[
		\widetilde{G_u}=\G_{a}^{0|1}.
		\]
		\item If $G=Q(n)$ and $u\in\g_{\ol{1}}$ with $u^2=0$ of rank $r$, then we have
		\[
		\widetilde{G_u}=Q(n-2r)\times\G_a^{0|r}.
		\]
		\item If $G=Q(n)$ and $u=T_{0,I_{n}}$, then we have
		\[
		\widetilde{G_u}=\G_a^{0|n}.
		\]
		\item If $G$ is split, i.e. $\g_{\ol{1}}$ is an odd abelian ideal of $\g$, then for any $u\in\g_{\ol{1}}$ we have that $\widetilde{G_u}$ is the supergroup with $(\widetilde{G_u})_0=C(u)_0$ and $\widetilde{\g_u}=\c(u)/[u,\g_{\ol{0}}]$, where $C(u)$ is the centralizer of $u$ in $G$, and $\c(u)$ is its Lie superalgebra.
		
	\end{enumerate}
\end{thm}

\begin{remark}
	We will see that in cases (1), (3), and (4), the odd abelian part $\G_a^{0|r}$ in $\widetilde{G_u}$ is naturally arising from the de Rham cohomology of a reductive algebraic group (and for us it will always be $GL(n)$.)
\end{remark}

\subsection{Relation to $\g_u$} Before beginning to prove Theorem \ref{main_thm_supergroups}, we briefly discuss a relationship between $\g_u$ and the Lie superalgebra $\widetilde{\g_u}=\operatorname{Lie}\widetilde{G_u}$.  

There is a natural map of Lie superalgebras $\g_u\to\widetilde{\g_u}$ defined as follows: since the adjoint vector field of $u_{ad}$ vanishes at $e$, $u_{ad}$ acts on $T_eG=\g$, and the action is exactly the adjoint action $\operatorname{ad}(u)$.  There is a natural well-defined map
\[
(T_eG)_u=\g_u\to T_e\widetilde{G_u}.
\]
Given a tangent vector $D$ one obtains a left-invariant vector field via the formula:
\[
(1\otimes D)\circ\Delta.
\]
Since the coproduct on $\C[G]_u$ is induced from the coproduct on $\C[G]$, our map of tangent spaces induces a map of Lie superalgebras $\g_u\to \operatorname{Lie}(\widetilde{G_u})$.

\begin{lemma}\label{inj sufficient}
	Suppose that there exists a faithful finite-dimensional $G$-module $V$ such that under the adjoint action of $\g$ we have $\End(V)=\g\oplus W$; then the map $\g_u\to\operatorname{Lie}(\widetilde{G_u})$ is injective.
\end{lemma}
\begin{proof}
	Write $\m_e$ for the maximal ideal of $e$; in every case we have an isomorphism $\m_e/\m_e^2\cong\g^*$ under the adjoint action.  We want to show that $\g^*$ splits off $\m_e$, and thus off $\C[G]$.  Under our assumptions, the pullback map $\End(V)^*\to\C[G]$ provides our splitting $\g^*\to\m_e$.
\end{proof}

Note that the hypotheses of Lemma \ref{inj sufficient} are preserved under taking subalgebras which are the fixed points of a semisimple automorphism. The hypotheses also hold for any superalgebra for which there exists a representation $V$ such that the supertrace form induces a nondegenerate form on $\g$.  Using this, one can check that the above property holds for $\g\l(m|n)$, $\o\s\p(m|2n)$, $\p(n)$, and $\q(n)$.

\subsection{Warm-up: $GL(1|1)$}
We begin with $G=GL(1|1)$, proving (2) in Theorem \ref{main_thm_supergroups}.

Write $\C[G]=\C[x^{\pm1},y^{\pm1},\xi,\eta]$ with $x,y$ even and $\xi,\eta$ odd.  The comultiplication map is:
\[
\Delta(x)=x\otimes x+\xi\otimes\eta,\ \ \ \Delta(y)=y\otimes y+\eta\otimes\xi
\]
\[
\Delta(\xi)=x\otimes\xi+\xi\otimes y, \ \ \ \Delta(\eta)=\eta\otimes x+y\otimes \eta
\]
Write $\g=\g\l(1|1)=\begin{bmatrix}I+h&E\\F&I-h\end{bmatrix}$ for its lie superalgebra, viewed as the tangent space at the identity.
For a tangent vector $w\in T_eG$, we write
\[
w_R=(\id\otimes w)\circ\Delta, \ \ \ \ w_L=-(w\otimes\id)\circ\Delta
\]
for the vector fields on $G$ corresponding the right, resp. left infinitesimal translation on $G$.  Being an open subset of affine space, the tangent bundle has a global trivialization given by the sections $\d_x,\d_y,\d_{\xi},$ and $\d_{\eta}$.  We use these to give another natural basis of $T_eG$ given by the elements $D_a=ev_e\circ\d_a$, where $a$ is some coordinate and $ev_e$ is evaluation at the identity.  Then we have
\[
I_R:=(D_x+D_y)_R=x\d_x+y\d_y+\xi\d_\xi+\eta\d_\eta,\ \ \ h_R:=(D_x-D_y)_L=x\d_x-y\d_y-\xi\d_\xi+\eta\d_\eta
\] 
\[
E_R:=(-D_{\xi})_R=\eta\d_y-x\d_\xi,\ \ \ \ F_R:=(D_{\eta})_R=-\xi\d_x+y\d_\eta
\]
and
\[
I_L:=(D_x+D_y)_L=-(x\d_x+y\d_y+\xi\d_\xi+\eta\d_\eta),\ \ \ h_L:=(D_x-D_y)_L=-x\d_x+y\d_y-\xi\d_\xi+\eta\d_\eta
\] 
\[
E_L:=(-D_{\xi})_L=\eta\d_x+y\d_\xi,\ \ \ \ F_L:=(D_{\eta})_L=-\xi\d_y-x\d_\eta
\]
Consider the vector field
\[
u=E+\lambda F=(E_R+E_L)+\lambda(F_{R}+F_{L})=(\eta-\lambda\xi)(\d_x+\d_y)+(y-x)(\d_\xi+\lambda\d_{\eta})
\] 
The square of this vector field is a multiple of the adjoint vector field of $I$, which is 0 since $I$ is central.  One can now compute by a spectral sequence argument (or directly by hand) that its cohomology is $\C\langle1,x^{-1}(\eta-\lambda\xi)\rangle$.  We can also see this from the module structure.  We recall from Sec. 6.4 of \cite{sherman2021spherical} that as a $\g\times\g$-module we have
\[
\C[G]=M_{0}\oplus \bigoplus_{\lambda\ typical}L(\lambda)\boxtimes L(\lambda)^*
\]
where $M_{0}$ has the following module structure
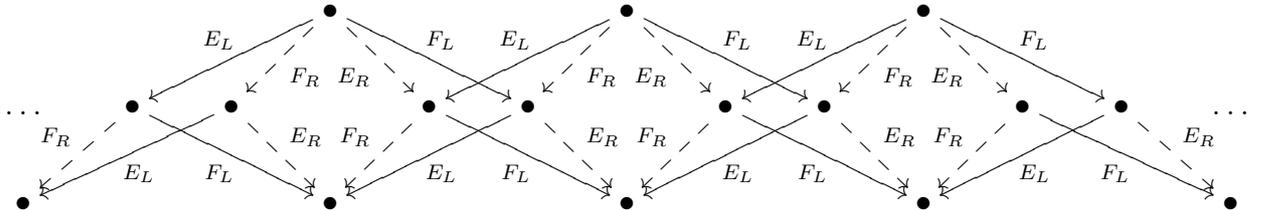
\[
\xymatrix{&&&\bullet \ar[dll]_{E_L}\ar[drr]^{F_L} \ar@{-->}[dl]^{F_R} \ar@{-->}[dr]_{E_R}&&&\bullet \ar[dll]_{E_L}\ar[drr]^{F_L} \ar@{-->}[dl]^{F_R} \ar@{-->}[dr]_{E_R}&&&\bullet \ar[dll]_{E_L}\ar[drr]^{F_L} \ar@{-->}[dl]^{F_R} \ar@{-->}[dr]_{E_R}&&&\\ \hdots&\bullet \ar[drr]_{F_L} \ar@{-->}[dl]_{F_R} & \bullet \ar[dll]^{E_L} \ar@{-->}[dr]^{E_R} & & \bullet\ar[drr]_{F_L} \ar@{-->}[dl]_{F_R} & \bullet \ar[dll]^{E_L} \ar@{-->}[dr]^{E_R} & &\bullet \ar[drr]_{F_L} \ar@{-->}[dl]_{F_R} & \bullet \ar[dll]^{E_L} \ar@{-->}[dr]^{E_R} & & \bullet\ar[drr]_{F_L} \ar@{-->}[dl]_{F_R} & \bullet \ar[dll]^{E_L} \ar@{-->}[dr]^{E_R} & \hdots \\\bullet & && \bullet & &&\bullet & & &\bullet & && \bullet & &&\bullet}
\]
Since the modules $L(\lambda)\otimes L(\lambda)^*$ for $\lambda$ typical are projective, we get no cohomology here.  On $M_{0}$ one sees we have a one-dimensional subspace in cohomology in the middle row (giving an odd vector) and a one-dimensional subspace from the bottom row (and even vector).  Indeed, any dot on the bottom row is equivalent in cohomology, but non-trivial.  In the middle row, one see that the dots come in pairs in our diagram; a diagonal subspace of these dots will be in the kernel of $E+{\lambda}F$, and will define a nontrivial class in cohomology.

\begin{remark}
	The case of $G=SL(1|1)$ can also be computed explicitly, and for a nonzero $u$ with $[u,u]=0$ one obtains $\widetilde{G_u}=\G_{a}^{0|2}$.
\end{remark}

\subsubsection{The Lie superalgebra}  Let us check that for $\g=\g\l(1|1)$ and $u=E+\lambda F$ as above, we have that $\operatorname{Lie}\widetilde{GL(1|1)_u}=\g(u,\C\langle I\rangle)=\C\langle F\rangle$.  Indeed, the action of $F$ by left or right infinitesimal translation on functions on $\widetilde{GL(1|1)_u}$ takes the odd generator $x^{-1}(\eta-\lambda\xi)$ to a non-zero multiple of 1, and annihilates 1.  

\subsection{The case of $GL(m|n)$} In this subsection we compute $\widetilde{GL(m|n)_u}$ for all $u$ with $[u,u]=0$.

Assume $1\leq m\leq n$.  Present the coordinate ring of $G=GL(m|n)$ as $\C[G]=\C[x_{ij},y_{ij},\det(x_{ij})^{-1},\det(y_{ij})^{-1},\xi_{ij},\eta_{ij}]$.  The coalgebra structure is given by
\[
\Delta(x_{ij})=\sum\limits_{k}x_{ik}\otimes x_{kj}+\sum\limits_{k}\xi_{ik}\otimes\eta_{kj}
\]
\[
\Delta(y_{ij})=\sum\limits_{k}y_{ik}\otimes y_{kj}+\sum\limits_{k}\eta_{ik}\otimes\xi_{kj}
\]
\[
\Delta(\xi_{ij})=\sum\limits_{k}x_{ik}\otimes\xi_{kj}+\sum\limits_{k}\xi_{ik}\otimes y_{kj}
\]
\[
\Delta(\eta_{ij})=\sum\limits_{k}\eta_{ik}\otimes x_{kj}+\sum\limits_{k}y_{ik}\otimes\eta_{kj}
\]
We compute that
\begin{eqnarray*}
	E_{\ell}& := &(D_{\xi_{\ell\ell}}\otimes 1-1\otimes D_{\xi_{\ell\ell}})\circ\Delta\\
	& =  &\sum\limits_{j=1}^{m}\eta_{\ell j}\d_{x_{\ell j}}+\sum\limits_{j=1}^{n}y_{\ell j}\d_{\xi_{\ell j}}+\sum\limits_{i=1}^{n}\eta_{i\ell}\d_{y_{i\ell}}-\sum\limits_{i=1}^{m}x_{i\ell}\d_{\xi_{i\ell}}
\end{eqnarray*}

\begin{eqnarray*}
	F_{\ell}& := &(D_{\eta_{\ell\ell}}\otimes 1-1\otimes D_{\eta_{\ell\ell}})\circ\Delta\\
	& =  &\sum\limits_{i=1}^{m}\xi_{i\ell}\d_{x_{i\ell}}-\sum\limits_{i=1}^{n}y_{i\ell}\d_{\eta_{i\ell}}+\sum\limits_{j=1}^{n}\xi_{\ell j}\d_{y_{\ell j}}+\sum\limits_{j=1}^{m}x_{\ell j}\d_{\eta_{\ell j}}
\end{eqnarray*}

Consider the vector field:
\[
u=E_1+\dots+E_r+F_{r+1}+\dots+F_s=D_1+D_2
\]
where $0\leq r\leq s\leq m$, and
\[
D_1=\sum\limits_{\ell=1}^{r}\left(\sum\limits_{j=1}^{n}y_{\ell j}\d_{\xi_{\ell j}}-\sum\limits_{i=1}^{m}x_{i\ell}\d_{\xi_{i\ell}}\right)+\sum\limits_{\ell=r+1}^{s}\left(\sum\limits_{j=1}^{m}x_{\ell j}\d_{\eta_{\ell j}}-\sum\limits_{i=1}^{n}y_{i\ell}\d_{\eta_{i\ell}}\right),
\]
and
\[
D_2=\sum\limits_{\ell=1}^{r}\left(\sum\limits_{j=1}^{m}\eta_{\ell j}\d_{x_{\ell j}}+\sum\limits_{i=1}^{n}\eta_{i\ell}\d_{y_{i\ell}}\right)+\sum\limits_{\ell=r+1}^{s}\left(\sum\limits_{i=1}^{m}\xi_{i\ell}\d_{x_{i\ell}}+\sum\limits_{j=1}^{n}\xi_{\ell j}\d_{y_{\ell j}}\right),
\]
Note then every element $v\in\g_{\ol{1}}$ with $[v,v]=0$ is conjugate to one of the above form.  We set, for $1\leq q\leq r\leq p\leq s$
\[
\alpha_{pq}=\xi_{pq}+\eta_{pq}, \ \ \ \ \beta_{qp}=\xi_{qp}+\eta_{qp}
\]
\[
\varphi_{pq}=\frac{1}{2}(\eta_{pq}-\xi_{pq}), \ \ \ \ \psi_{qp}=\frac{1}{2}(\xi_{qp}-\eta_{qp})
\]
Then we may write $D_1$ as:
\begin{eqnarray*}
	D_1& = &\sum\limits_{1\leq\ell,k\leq r}(y_{\ell k}-x_{\ell k})\d_{\xi_{\ell k}}+\sum\limits_{r+1\leq \ell,k\leq s}(y_{\ell k}-x_{\ell k})\d_{\eta_{\ell k}}\\
	& + &\sum\limits_{1\leq q\leq r<p\leq s}\left[x_{pq}\d_{\varphi_{pq}}+y_{qp}\d_{\psi_{qp}}\right]\\
	& + &\sum\limits_{\ell=1}^{r}\left(\sum\limits_{j=s+1}^{n}y_{\ell j}\d_{\xi_{\ell j}}-\sum\limits_{i=s+1}^{m}x_{i\ell}\d_{\xi_{i\ell}}\right)\\
	& + &\sum\limits_{\ell=r+1}^{s}\left(\sum\limits_{j=s+1}^{m}x_{\ell j}\d_{\eta_{\ell j}}-\sum\limits_{i=s+1}^{n}y_{i\ell}\d_{\eta_{i\ell}}\right)
\end{eqnarray*}
And now $D_2$:
\begin{eqnarray*}
	D_2& = &\sum\limits_{1\leq\ell,k\leq r}\eta_{\ell k}(\d_{x_{\ell k}}+\d_{y_{\ell k}})+\sum\limits_{r+1\leq \ell,k\leq s}\xi_{\ell k}(\d_{x_{\ell k}}+\d_{y_{\ell k}})\\
	& + &\sum\limits_{1\leq q\leq r<p\leq s}\left[\beta_{qp}\d_{x_{qp}}+\alpha_{pq}\d_{y_{pq}}\right]\\
	& + &\sum\limits_{\ell=1}^{r}\left(\sum\limits_{j=s+1}^{m}\eta_{\ell j}\d_{x_{\ell j}}+\sum\limits_{i=s+1}^{n}\eta_{i\ell}\d_{y_{i\ell}}\right)\\
		& + &\sum\limits_{\ell=r+1}^{s}\left(\sum\limits_{i=s+1}^{m}\xi_{i\ell}\d_{x_{i\ell}}+\sum\limits_{j=s+1}^{n}\xi_{\ell j}\d_{y_{\ell j}}\right)\\
\end{eqnarray*}
We claim that $D_1$ and $D_2$ supercommute.  Indeed, the above terms are grouped according to their indices, and so one only needs to check the following supercommutators are zero:
\[
\left[\sum\limits_{1\leq\ell,k\leq r}(y_{\ell k}-x_{\ell k})\d_{\xi_{\ell k}}+\sum\limits_{r+1\leq \ell,k\leq s}(y_{\ell k}-x_{\ell k})\d_{\eta_{\ell k}},\sum\limits_{1\leq\ell,k\leq r}\eta_{\ell k}(\d_{x_{\ell k}}+\d_{y_{\ell k}})+\sum\limits_{r+1\leq \ell,k\leq s}\xi_{\ell k}(\d_{x_{\ell k}}+\d_{y_{\ell k}})\right]
\]
\[
\left[\sum\limits_{1\leq q\leq r<p\leq s}\left[x_{pq}\d_{\varphi_{pq}}+y_{qp}\d_{\psi_{qp}}\right],\sum\limits_{1\leq q\leq r<p\leq s}\left[\beta_{qp}\d_{x_{qp}}+\alpha_{pq}\d_{y_{pq}}\right]\right]
\]
\[
\left[\sum\limits_{\ell=1}^{r}\left(\sum\limits_{j=s+1}^{n}y_{\ell j}\d_{\xi_{\ell j}}-\sum\limits_{i=s+1}^{m}x_{i\ell}\d_{\xi_{i\ell}}\right),\sum\limits_{\ell=1}^{r}\left(\sum\limits_{j=s+1}^{m}\eta_{\ell j}\d_{x_{\ell j}}+\sum\limits_{i=s+1}^{n}\eta_{i\ell}\d_{y_{i\ell}}\right)\right]\]
\[
\left[\sum\limits_{\ell=r+1}^{s}\left(\sum\limits_{j=s+1}^{m}x_{\ell j}\d_{\eta_{\ell j}}-\sum\limits_{i=s+1}^{n}y_{i\ell}\d_{\eta_{i\ell}}\right),\sum\limits_{\ell=r+1}^{s}\left(\sum\limits_{i=s+1}^{m}\xi_{i\ell}\d_{x_{i\ell}}+\sum\limits_{j=s+1}^{n}\xi_{\ell j}\d_{y_{\ell j}}\right)\right]
\]
In all cases one sees the supercommutators are zero by direct inspection.

We now assign a $\Z\times\Z$-grading to $\C[G]$ so that $D_1$ will be an operator of degree $(1,0)$ and $D_2$ will be an operator of degree $(0,1)$, so we obtain a double complex on $\C[G]$, allowing us to use a spectral sequence to compute our cohomology.

We set $x_{ij},y_{ij}$ to have degree 0.  Now we will set all the odd coordinates in the above expression for $D_1$ to be of degree $(-1,0)$, the odd coordinates in $D_2$ will have degree $(0,1)$, and the remaining odd coordinates will have degree 0.  Explicitly, we declare that $\varphi_{pq},\psi_{qp}$ have degree $(-1,0)$; $\xi_{ij}$ has degree $(-1,0)$ if $1\leq i,j\leq r$, $1\leq i\leq r$ and $s+1\leq j\leq n$, or $s+1\leq i\leq m$ and $1\leq j\leq r$; and finally $\eta_{ij}$ has degree $(-1,0)$ if $r+1\leq i,j\leq s$, $r+1\leq i\leq s$ and $s+1\leq j\leq m$, or $s+1\leq i\leq n$ and $r+1\leq j\leq s$.

On the other hand we declare that $\alpha_{pq},\beta_{qp}$ have degree $(0,1)$; $\xi_{ij}$ has degree $(0,1)$ if $r+1\leq i,j\leq s$, $r+1\leq i\leq s$ and $s+1\leq j\leq n$, or $s+1\leq i\leq m$ and $r+1\leq j\leq s$; and $\eta_{ij}$ has degree $(0,1)$ if $1\leq i,j\leq r$, $1\leq i\leq r$ and $s+1\leq j\leq m$, or $s+1\leq i\leq n$ and $1\leq j\leq r$.

We finally declare to have degree 0 the remaining odd coordinates $\xi_{ij},\eta_{ij}$ which do not lie in the linear span of the coordinates for whom we have already given a degree.

Now we have obtained a double complex; we take cohomology with respect to $D_1$ first; since $D_1$ determines a Koszul complex on regular elements, its cohomology is concentrated in degree $(0,\ell)$, and so the cohomology of the first page will be the full cohomology we are looking for.  Now after taking cohomology with respect to $D_1$, we obtain the algebra with generators which we divide into three sets: the first is the coordinate algebra of a copy of $\C[GL(m-s|n-s)]$, which itself centralizes $u$:
\[
\{x_{ij}\}_{s+1\leq i,j\leq m}\cup\{y_{ij}\}_{s+1\leq i,j\leq n}\cup\{\xi_{ij}\}_{\substack{s+1\leq i\leq m,\\s+1\leq j\leq n}}\cup\{\eta_{ij}\}_{\substack{s+1\leq i\leq n,\\s+1\leq j\leq m}},
\]
and we have $\det(x_{ij})$ for $s+1\leq i,j\leq m$, $\det(y_{ij})$, for $s+1\leq i,j\leq n$ are invertible.  The next set of coordinates we will think of as giving the cotangent bundle on $GL(r)\times GL(s-r)$; here the even coordinates $z_{ij}$ are equal to the image of $x_{ij}$ and $y_{ij}$ after taking cohomology with respect to $D_1$:
\[
\{z_{ij},\xi_{ij}\}_{1\leq i,j\leq r}\cup \{z_{ij},\eta_{ij}\}_{r+1\leq i,j\leq s}
\] 
and we have $\det(z_{ij})$ for $1\leq i,j\leq r$ and $\det(z_{ij})$ for $r+1\leq i,j\leq s$ are invertible.  Note that this $GL(r)\times GL(s-r)$ is coming from a natural subalgebra $\g\l(r)\times\g\l(s-r)\sub[u,\g]$.

The remaining set of coordinates we will view as the cotangent bundle on an affine space:
\[
\{y_{pq},\alpha_{pq}\}_{1\leq q\leq r<p\leq s}\cup\{x_{qp},\beta_{qp}\}_{1\leq q\leq r<p\leq s}
\]
\[
\{x_{ij},\eta_{ij}\}_{\substack{1\leq i\leq r\\ s+1\leq j\leq m}}\cup \{y_{ij},\eta_{ij}\}_{\substack{s+1\leq i\leq n\\ 1\leq j\leq r}}
\]
\[
\{x_{ij},\xi_{ij}\}_{\substack{s+1\leq i\leq m\\ r+1\leq j\leq s}}\cup \{y_{ij},\xi_{ij}\}_{\substack{r+1\leq i\leq s\\ s+1\leq j\leq n}}.
\]
To help the reader, here is graphic of where we have which coordinates:
\[
\begin{bmatrix}z & x & x & 0 & \beta & 0 \\0 & z & 0 & \alpha & \xi & \xi \\0 & x & x & 0 & \xi & \xi\\ \eta & \beta & \eta & z & 0 & 0\\\alpha & 0 & 0& y & z & y\\ \eta & 0 & \eta& y & 0 & y\end{bmatrix}
\]
Now we need to take cohomology of this algebra with respect to the odd vector field $D_2$.  In our new algebra this operator looks like:
\begin{eqnarray*}
	D_2& = &\sum\limits_{1\leq\ell,k\leq r}\eta_{\ell k}\d_{x_{\ell k}}+\sum\limits_{r+1\leq \ell,k\leq s}\xi_{\ell k}\d_{x_{\ell k}}\\
	& + &\sum\limits_{1\leq q\leq r<p\leq s}(\beta_{qp}\d_{x_{qp}}+\alpha_{pq}\d_{y_{pq}})\\
	& + &\sum\limits_{\ell=1}^{r}\left(\sum\limits_{j=s+1}^{m}\eta_{\ell j}\d_{x_{\ell j}}+\sum\limits_{i=s+1}^{n}\eta_{i\ell}\d_{y_{i\ell}}\right)\\
	& + &\sum\limits_{\ell=r+1}^{s}\left(\sum\limits_{i=s+1}^{m}\xi_{i\ell}\d_{x_{i\ell}}+\sum\limits_{j=s+1}^{n}\xi_{\ell j}\d_{y_{\ell j}}\right)\\
\end{eqnarray*}  

Now the above differential is linear over our copy of $\C[GL(m-s|n-s)]$, which as we already described has coordinates $x_{ij},y_{ij},\xi_{ij},\eta_{ij}$ where $s+1\leq i\leq m$ and $s+1\leq j\leq n$.   In fact, if write $A$ for the coordinate algebra of the above coordinates, then we have a natural map $\Spec A\to GL(m-s|n-s)$, and $D_2$ is exactly the de Rham differential of $\Spec A$ over $GL(m-s|n-s)$.  Thus the cohomology is just the tensor product of the algebra of functions on $GL(m-s|n-s)$ with the algebraic de Rham cohomology of $GL(r)\times GL(s-r)\times \A$, where $\A$ is some affine space.  As an algebra we find therefore that the cohomology is
\[
\C[GL(m-s|n-s)]\otimes H_{dR}^*(GL(r)\times GL(s-r))
\]
We know that $H_{dR}^*(GL(r)\times GL(s-r))$ is a Grassmann algebra on $s$ generators.  From this, we can see the group structure will be:
\[
\widetilde{GL(m|n)_u}=GL(m-s|n-s)\times\G_{a}^{0|s}.
\]

\subsubsection{One full rank case}\label{full rank gl computation example}  To see the above computation in a simple case, we consider a special vector field for the case $m=n$.  Let
\[
u=E_{1}+\dots+E_{n}=\sum \eta_{ij}(\d_{x_{ij}}+\d_{y_{ij}})+\sum\limits (y_{ij}-x_{ij})\d_{\xi_{ij}}
\]
Using the spectral sequence argument again as above, one computes then in this case that the cohomology is nothing but the De Rham cohomology of $GL_n$ embedded diagonally, a Grassmann algebra on $n$ generators, giving a purely odd supergroup.  We will come back to this computation from a different perspective in the last section. 

\subsection{$Q(n)$ Case}  We now prove (3) of Theorem \ref{main_thm_supergroups}.  We present $\C[Q(n)]$ as $\C[x_{ij},\xi_{ij},\det(x_{ij})^{-1}]$.  We have
\[
\Delta(x_{ij})=\sum\limits_k(x_{ik}\otimes x_{kj}+\xi_{ik}\otimes\xi_{kj})
\]
\[
\Delta(\xi_{ij})=\sum\limits_k(\xi_{ik}\otimes x_{kj}+x_{ik}\otimes\xi_{kj}).
\]
Thus
\begin{eqnarray*}
	E_{i}& := &(D_{\xi_{i,i+1}}\otimes 1-1\otimes D_{\xi_{i,i+1}})\circ\Delta\\
	& = &\sum\limits_k(\xi_{i+1,k}\d_{x_{i,k}}+\xi_{k,i}\d_{x_{k,i+1}})+\sum\limits_{k}(x_{i+1,k}\d_{\xi_{i,k}}-x_{k,i}\d_{\xi_{k,i+1}})\\
	& = &\sum\limits_{k\neq i,i+1}(\xi_{i+1,k}\d_{x_{i,k}}+\xi_{k,i}\d_{x_{k,i+1}})+\sum\limits_{k\neq i,i+1}(x_{i+1,k}\d_{\xi_{i,k}}-x_{k,i}\d_{\xi_{k,i+1}})\\
	& + &(x_{i+1,i+1}-x_{ii})\d_{\xi_{i,i+1}}+x_{i+1,i}(\d_{\xi_{i,i}}-\d_{\xi_{i+1,i+1}})\\
	& + &(\xi_{i+1,i+1}+\xi_{i,i})\d_{x_{i,i+1}}+\xi_{i+1,i}(\d_{x_{i,i}}+\d_{x_{i+1,i+1}}).
\end{eqnarray*}
It follows that
\[
u=E_1+E_3+\dots+E_{2r-1}=D_1+D_2,
\]
where
\begin{eqnarray*}
	D_1& = &\sum\limits_{1\leq s,t\leq r}(x_{2s,2t}-x_{2s-1,2t-1})\d_{\xi_{2s-1,2t}}+\sum\limits_{1\leq s,t\leq r}x_{2s,2t-1}(\d_{\xi_{2s-1,2t-1}}-\d_{\xi_{2s,2t}})\\
	& + &\sum\limits_{1\leq s\leq r, k>2r}x_{2s,k}\d_{\xi_{2s-1,k}}-x_{k,2s-1}\d_{\xi_{k,2s}},
\end{eqnarray*}
and
\begin{eqnarray*}
	D_2& = &\sum\limits_{1\leq s,t\leq r}(\xi_{2s,2t}+\xi_{2s-1,2t-1})\d_{x_{2s-1,2t}}+\sum\limits_{1\leq s,t\leq r}\xi_{2s,2t-1}(\d_{x_{2s-1,2t-1}}+\d_{x_{2s,2t}})\\
	& + &\sum\limits_{1\leq s\leq r, k>2r}\xi_{2s,k}\d_{x_{2s-1,k}}+\xi_{k,2s-1}\d_{x_{k,2s}}
\end{eqnarray*}
and since $[D_1,D_1]=[D_2,D_2]=0$, we have $[D_1,D_2]=0$.  Note that every $v\in\q(n)_{\ol{1}}$ with $[v,v]=0$ is conjugate to an element of the above form. Working in the same way as for $GL(m|n)$ we can again create a double complex such that $D_1$ has degree (1,0) and $D_2$ has degree (0,1). The operator $D_1$ gives a Koszul complex; its cohomology will be
\[
\C[y_{st},\det(y_{st})^{-1},x_{2s-1,2t},x_{2s-1,k},x_{k,2s},\xi_{2s,k},\xi_{k,2s-1},\eta_{s,t},\xi_{2s,2t-1},x_{ij},\det(x_{ij})^{-1},\xi_{ij}]
\]
where $1\leq s,t\leq r$, and $2r<i,j,k\leq n$.  Here $y_{st}$ is the image of $x_{2s,2t}=x_{2s-1,2t-1}$, and $\eta_{s,t}$ is the image of $\xi_{2s,2t}=\xi_{2s-1,2t-1}$.  Under these coordinates, $D_2$ becomes:
\[
D_2=\sum\limits_{1\leq s,t\leq r}(\eta_{st}\d_{x_{2s-1,2t}}+\xi_{2s,2t-1}\d_{y_{st}})+\sum\limits_{k}\xi_{2s,k}\d_{x_{2s-1,k}}+\xi_{k,2s-1}\d_{x_{k,2s}}.
\]
This is defining a de Rham complex over $\C[Q(n-2r)]$, and its cohomology is given by
\[
H_{dR}(GL(r|r))\otimes\C[Q(n-2r)].
\]
The group structure is given by:
\[
\widetilde{Q(n)_u}=Q(n-2r)\times\G_{a}^{0|r}.
\]

\subsubsection{Another special case for $Q(n)$}\label{special q case}  We now show part (4) of Theorem \ref{main_thm_supergroups}.  Consider the element of $\q(n)$ given by
\[
u=\begin{bmatrix} 0 & I_n \\ I_n & 0 \end{bmatrix}
\]

Then $u^2$ is the identity matrix, which is central and thus the adjoint vector field of $u$ acts as a square-zero operator on $\C[Q(n)]$.  Further, $u$ is $G_0$-invariant, and thus its adjoint vector field everywhere vanishing.  Explicitly, it defines the vector field
\[
\sum\limits_{i,j}\xi_{ij}\d_{x_{ij}}
\]
and thus gives the de Rham complex on $GL(n)$.  Thus
\[
\C[Q(n)]_u=H_{dR}^\bullet(GL(n)).
\]
In particular, $\widetilde{Q(n)_u}=\G_a^{0|n}$.    

\begin{remark}
	Another, albeit less explicit, approach to this computation is explained in Prop.~5.9 of \cite{localization}, and works as follows: observe that for any point $g\in Q(n)_0$, $u_{ad}$ induces an isomorphism:
	\[
	[u,-]:(T_{g}G)_{\ol{1}}\to (T_gG)_{\ol{0}}
	\]
	Thus, by Lem.~5.7 of \cite{localization}, there is an isomorphism of $Q(n)$ with the cotangent bundle of $Q(n)_0=GL(n)$ such that $u$ becomes the de Rham differential.  Note that an analytic version of this identification of the cotagent bundle is given in \cite{vaintrob1996normal}.
\end{remark}

\subsection{Odd Abelian Case}
Let $G$ be a linear algebraic supergroup and assume that $\g_{\ol{1}}$ is abelian; such supergroups are called `split'.  We seek to show part (5) of Theorem \ref{main_thm_supergroups}.  First of all, any such supergroup may be presented as follows: 
\[
\C[G]=\C[G_0]\otimes S\g_{\ol{1}}^*
\]
where $\Delta(f)=\Delta_0(f)$ for $f\in\C[G_0]$, and the comultiplication on $\g_{\ol{1}}^*$ is as follows.  Write $a:\g_{\ol{1}}\to\C[G_0]\otimes\g_{\ol{1}}$ for the comodule structure morphism for the action of $G_0$ on $\g_{\ol{1}}$.  If $u_1,\dots,u_n$ is a basis of $\g_{\ol{1}}$, then we may write
\[
a(u_i)=\sum\limits_{j}f_{ij}\otimes u_j.
\] 
Write $\xi_1,\dots,\xi_n$ for the dual basis in $\g_{\ol{1}}^*$; then we have
\[
\Delta(\xi_i)=\xi_{i}\otimes 1+\sum\limits_{j}f_{ji}\otimes\xi_j.
\]
Now we may assume that $u=u_1$; then we have:
\[
u_{ad}=(f_{11}-1)\d_{\xi_1}+\sum\limits_{i\geq 2}f_{1i}\d_{\xi_i}.
\]
Thus $u_{ad}$ is defining a Koszul complex on $G$ for the elements $f_{11}-1,f_{21},\dots,f_{1n}$; further, these elements generate exactly the ideal of $C(u)_0$ in $G_0$.  Thus in particular this cohomology will be graded, meaning that the supergroup $\widetilde{G_u}$ will be again be split, i.e. $(\widetilde{\g_u})_{\ol{1}}$ will be abelian.  Further we will have $(\widetilde{G_u})_{0}=C(u)_0$.  Therefore it remains to determine the structure of $\widetilde{\g_u}$ as a $C(u)_0$-module.

However, by Prop. 1.28 of \cite{sanna2014rational}, because $G_0$ is smooth and $C(u)_{0}$ is smooth subvariety, this Koszul complex has cohomology given by $\bigwedge\EE$, where $\EE$ is the vector bundle on $C(u)_0$ defined by the following short exact sequence:
\[
0\to\EE\to \OO_{C(u)_0}\otimes\g_{\ol{1}}^*\to \NN_{C(u)_0/G_0}^\vee\to 0.
\]
As a matter of explanation, $\NN_{C(u)_0/G_0}^\vee=\II_{C(u_0)}/\II_{C(u_0)}^2$ is the conormal bundle of $C(u)_0$ in $G_0$, and the last map is induced by the natural map $\OO_{G_0}\otimes\g_{\ol{1}}^*\to\II_{C_0}$ coming from the application of $u_{ad}$.  However, restricting to the fiber of the above short exact sequence at the identity, we obtain
\[
0\to\EE|_{eG}\to \g_{\ol{1}}^*\to (\g_{\ol{0}}/\c(u))^*\to 0
\]
where the last map is dual to the map given by $[u,-]:\g_{\ol{0}}/\c(u)\to\g_{\ol{1}}$. Thus we have that $\EE|_0\cong (\g_{\ol{1}}/[u,\g_{\ol{0}}])^*$.    Now all the above sequences are $C(u)_0$-equivariant, and thus it follows that as a $C(u)_0$-module we have $(\widetilde{\g_u})_{\ol{1}}\cong \g_{\ol{1}}/[u,\g_{\ol{0}}]$, completing the proof.


\section{Third approach: looking in the enveloping algebra}

A less developed approach that probably leads to rich symmetries is to look in the full enveloping algebra.  Again let $\k\sub[u,\g^c]$ be a subalgebra which acts semisimply on all modules in $\CC$.  Then once again by Lemma \ref{invariance_ds_commute}, we have a canonical isomorphism $DS_u\cong DS_u\circ(-)^{\k}$.  For $M$ in $\CC$, $(\UU\g)^{\k}$ will act on $M^{\k}$, with the subspace $(\k\UU\g+\UU\g\k)^{\k}$ acting trivially.  Write $\II$ for the ideal generated by $(\k\UU\g+\UU\g\k)^{\k}$ in $(\UU\g)^{\k}$, so that $(\UU\g)^{\k}/\II$ naturally acts on $M^{\k}$.  Then the associative superalgebra
\[
\UU(\g,\k,u):=\left(\frac{(\UU\g)^{\k}}{\II}\right)_u
\]
will naturally act on $(M^{\k})_u$, and therefore define an action on the functor $DS_u$. 

\begin{remark}
	It is not apparent to this author whether $\UU(\g,\k,u)$ is a bialgebra, much less a Hopf algebra, and while it acts naturally on $DS_u$ as a functor, it is not clear how to state its relationship with tensor product.
\end{remark} 

\begin{example}
	Let $u\in\g\l(1|1)_{\ol{1}}$, and consider $\k=\C\langle c\rangle$.  Then because $c$ is central, $(\UU\g)^{\k}/\II=\UU\p\g\l(1|1)$.  It follows that $\UU(\g,\k,u)=\UU\g(\k,u)$ in this case.
\end{example}

\subsection{An example for $\g\l(n|n)$}  The following example is taken from sections 26 and 27 of \cite{HW}.  Consider $\g=\g\l(n|n)$, and let 
\[
u=\begin{bmatrix} 0 & I_n\\0 & 0\end{bmatrix}.
\]
Let $\g=\g_{-1}\oplus\g_{0}\oplus\g_1$ be the usual $\Z$-grading on $\g$.  Let $\k$ be the diagonal subalgebra of $\g_{0}$ given by
\[
\k=\begin{bmatrix}A & 0\\0 & A\end{bmatrix},
\]
where $A$ is arbitrary.   Then $[u,\g_{-1}]=\k\sub[u,\g]$, and thus $u$ defines an odd $\k$-equivariant isomorphism $\g_{-1}\to\k$.  Now let $L_0$ be a $\g_{0}\cong\g\l(n)\times\g\l(n)$-module, and consider the Kac-module $K(L_0):=\Ind_{\g_{0}\oplus\g_{1}}^{\g}L_0$.  Over $\g_{-1}\oplus\k$, this module is given by
\[
\bigwedge\g_{-1}\otimes \operatorname{Res}_{\k}L_0
\]
\begin{lemma}[\cite{HW}] Under the isomorphism $\g_{-1}\cong\k$, the operator $u$ induces the Lie algebra homology differential on the Chevalley-Eilenberg complex $\bigwedge\g_{-1}\otimes \operatorname{Res}_{\k}L_0$.
\end{lemma}
It follows from \cite{hochschild1953cohomology} that
\[
DS_u\Ind_{\g_{0}\oplus\g_{1}}^{\g}L_0\cong H^\bullet(\k,\C)\otimes L_0^\k.
\]
Here $H^\bullet(\k,\C)$ denotes the Lie algebra homology of the trivial $\k$-module; in \cite{hochschild1953cohomology}, they further show that this homology is exactly given by $(\bigwedge\k)^{\k}$, and is isomorphic to an exterior algebra on $n$ generators, where $n$ is the rank of $\k\cong\g\l(n)$.

Now consider the associative superalgebra $\UU(\g,\k,u)$.  We have $(\UU\g_{-1})^{\k}\sub(\UU\g)^{\k}$, and therefore we have a natural map to the quotient:
\[
(\UU\g_{-1})^{\k}\to\frac{(\UU\g)^{\k}}{(\k\UU\g+\UU\g\k)^{\k}}.
\]
Call $R$ the image of this subalgebra. We have $[u,\UU\g_{-1}]\sub\UU\g_{-1}+(\k\UU\g+\UU\g\k)$, so that $u$ preserves $R$, and we obtain an action of $R_u$ on $(V^{\k})_u$.  However, if we take a Kac-module on the trivial module, $\Ind_{\g_{0}\oplus\g_{1}}^{\g}\C$ as defined above,  from what we showed earlier we will have that $DS_u\Ind_{\g_{0}\oplus\g_{1}}^{\g}\C$ will be a free module over $(\UU\g_{-1})^{\k}$.  It follows that 
\[
R\cong R_u\cong (\UU\g_{-1})^{\k},
\]
and it acts on $DS_u$.   Further we have completely described its action on $DS_u$ applied to any Kac module.

In the last section we will show that $R$ is naturally identified with the Lie superalgebra of a naturally constructed Lie supergroup which acts on $DS_u$ from Example \ref{full rank gl computation example}.

\subsection{General statement}\label{section general enveloping}
 Let $\g$ be a finite-dimensional Lie superalgebra with a compatible $\Z$-grading $\g=\g_{-1}\oplus\g_{0}\oplus\g_1$.  Suppose that $u\in\g_{1}$ is an element such that $[u,-]:\g_{-1}\to \g_0$ is injective and $\k:=[u,\g_{-1}]\sub[u,\g]$ is a reductive Lie algebra.  Then in particular $u$ defines an odd $\k$-equivariant isomorphism $\g_{-1}\cong\k$.
 
 Let $L_0$ be a finite-dimensional $\g_0$ which is semisimple over $\k$, and consider the associated Kac-module $K(L_0)=\Ind_{\g_0\oplus\g_{1}}^{\g}L_0$.

\begin{lemma} Under the isomorphism $\g_{-1}\cong\k$, the operator $u$ induces the Lie algebra homology differential on the Chevalley-Eilenberg complex $\bigwedge\g_{-1}\otimes \operatorname{Res}_{\k}L_0$.  In particular we have an identification
	\[
	DS_uK(L_0)=(\UU\g_{-1})^{\k}\otimes L_0^{\k}.
	\]
 \end{lemma}

\begin{proof}
The proof works verbatim as in Lemma 26.1 of \cite{HW}.
\end{proof}

Now $\g_{-1}\sub\g$ is an abelian subalgebra stable under $\k$, and thus we may consider the image $R$ of the map
\[
(\UU\g_{-1})^{\k}\hookrightarrow\frac{(\UU\g)^{\k}}{(\k\UU\g+\UU\g\k)^{\k}}.
\]
Again we have $[u,\UU\g_{-1}]\sub\UU\g_{-1}+(\k\UU\g+\UU\g\k)$, so that $u$ preserves $R$, and we obtain an action of $R_u$ on the functor $DS_u$.   The following proof again works the same way as in the $\g=\g\l(n|n)$ case.

\begin{thm}\label{thm_kac_mod}
We have 
\[
R_u\cong R\cong(\UU\g_{-1})^{\k}\cong (\bigwedge \k)^{\k},
\]
and $R\to\UU(\g,\k,u)$ is injective.  As a module over $R$, we have an isomorphism
\[
DS_uK(L_0)\cong R\otimes L_0^{\k},
\]
i.e., $DS_uK(L_0)$ is a free $R$-module of rank $\dim L_0^{\k}$.
\end{thm}

For a conjectural extension of the above theorem, see \cref{conj_intro_kac}.

\subsection{Extension to $\p(n)$}  The ideas and results of Section \ref{section general enveloping} in particular apply to $\g=\p(n)$. There are two cases here.  We recall that $\p(n)$ can be presented as the subalgebra of $\g\l(n|n)$ consisting of matrices of the form
\[
\begin{bmatrix}
A & B\\ C& -A^{t}
\end{bmatrix}
\]
where $B^t=B$ and $C^t=-C$.  This superalgebra admits a compatible $\Z$-grading $\p(n)=\p_{-1}\oplus\p_0\oplus\p_1$. Thus, given a $\g_0\cong\g\l(n)$-module $L_0$, we may define the thin and thick Kac modules as follows: the thin Kac module on $L_0$ is given by $K^-(L_0)=\Ind_{\g_{0}\oplus\g_{1}}^{\g}L_0$, and the thick Kac module on $L_0$ is given by $K^+(L_0)=\Ind_{\g_0\oplus\g_{-1}}^{\g}L_0$.

For any $n>0$ set
\[
u_+=\begin{bmatrix}
0 & I_n\\0 & 0
\end{bmatrix}.
\]
Then $\k_{+}:=[u_{+},\g_{-1}]\cong\s\o(n)$, and $u_{+}$ defines an odd isomorphism of $\k_{+}$-modules $\g_{-1}\cong\k_+$.   For $n$ even, set 
\[
u_-=\begin{bmatrix}
0 & 0\\ J & 0
\end{bmatrix}
\]
where 
\[
J=\begin{bmatrix} 0 & 1 & 0 & \dots& & 0\\ -1 & 0 & 1 & & & \vdots\\ 0  & -1 & 0  &  & &\\ \vdots & &&\ddots && \end{bmatrix}.
\]
Then in this case, $\k_{-}:=[u_-,\g_{1}]\cong\s\p(n)$, and $u_{-}$ defines an odd isomorphism of $\k_{-}$-modules $\g_{1}\cong\k_-$.  From Section \ref{section general enveloping} we have the following lemma and theorem:
\begin{lemma}
	For a $\g\l(n)$-module $L_0$, $u_{\pm}$ induces the Lie algebra homology differential on the Chevalley-Eilenberg complex 
	\[
	K^{\mp}(L_0)=\bigwedge\g_{\mp1}\otimes \operatorname{Res}_{\k_{\pm}}L_0
	\]
\end{lemma}

\begin{thm}\label{DS computation p case}
	\begin{enumerate}
	\item $\UU(\p_{-1})^{\k_+}$ is isomorphic to an exterior algebra on $\lfloor n/2\rfloor$ generators, and the natural map $\UU(\p_{-1})^{\k_+}\to\UU(\p(n),\k_+,u_+)$ is injective.
	
	\item  For $\g=\p(2n)$, $\UU(\p_{1})^{\k_-}$ is isomorphic to an exterior algebra on $n$ generators, and the natural map $\UU(\p_{1})^{\k_-}\to\UU(\p(n),\k_-,u_-)$ is injective.
	
	\item For a finite-dimensional $\g_{0}$-module $L_0$, $DS_{u_+}K^-(L_0)\cong \UU(\p_{-1})^{\k_+}\otimes L_0^{\k_+}$ as a module over $\UU(\p_{-1})^{\k_+}$, and $DS_{u_-}K^+(L_0)\cong\UU(\p_{1})^{\k_-}\otimes L_0^{\k_-}$ as a module over $\UU(\p_{1})^{\k_-}$.
\end{enumerate}
\end{thm}

\begin{remark}
	We recall that $\k_+=\s\o(n)\sub\g\l(n)$ and $\k_-\s\p(2n)\sub\g\l(2n)$ are spherical subalgebras, meaning that given an irreducible $\g\l(n)$-representation $L_0$, $\dim L_0^{\k_{\pm}}\leq 1$, and therefore $DS_{u_{\pm}}K^{\mp}(L_0)$ will either be 0 or isomorphic to a free module of rank one over $\UU(\p_{\mp})^{\k_{\pm}}$.  For the convenience of the reader, we recall which irreducible $\g\l(n)$-modules have either an $\s\o(n)$-invariant or an $\s\p(2n)$-invariant.  
	
	Label the irreducible representations of $\g\l(n)$ according to their dominant weights with respect to the usual positive system, meaning they are given by a decreasing sequence of integers $\lambda_1\geq\lambda_2\geq\dots\geq \lambda_n$.  Write $\lambda=(\lambda_1,\dots,\lambda_n)$, and let $L(\lambda)$ denote the irreducible representation of highest weight $\lambda$.  Then $L(\lambda)$ admits an $\s\o(n)$-invariant if and only if $\lambda_i-\lambda_{i+1}$ is even for all $1\leq i\leq n-1$.  For $\g\l(2n)$, $L(\lambda)$ admits an $\s\p(2n)$-invariant if and only if $\lambda_1=\lambda_2,\lambda_3=\lambda_4,\dots,\lambda_{2n-1}=\lambda_{2n}$.  
	
	We further note, as was proven in lem. 3.4.1 of \cite{balagovic2016translation}, that $K^{-}(L_0)$ is an irreducible $\p(n)$-module whenever the highest weight $\lambda$ of $L(\lambda)$ has that $\lambda_1>\lambda_2>\dots>\lambda_n$.  Thus for these irreducible representations we have now given a formula for the result of $DS_{u_+}$ on them.
\end{remark}

%
%
%
%

\section{Explicit realization of the odd generators} 

Let $G$ be a Lie supergroup with Lie superalgebra $\g$ admitting a $\Z$-grading $\g=\g_{-1}\oplus\g_0\oplus\g_{1}$, and such that $G_0$ preserves this grading.  Suppose as in Section \ref{section general enveloping} that $u\in\g_{1}$ is an element such that $[u,-]:\g_{-1}\to \g_0$ is injective and $\k:=[u,\g_{-1}]\sub[u,\g]$ is a reductive Lie algebra.  

Recall that we have an identification (due to Koszul in \cite{koszul1982graded}):
\[
\C[G]=\Hom_{\g_{\ol{0}}}(\UU\g,\C[G_0]).
\]
Choose the ordered basis of $\g_{\ol{1}}$ given by first taking a basis of $\g_{1}$ and then taking a basis of $\g_{-1}$.  Then using monomials ordered with respect to this basis we have a natural copy of $S\g_{\ol{1}}$ in $\UU\g$, giving rise to an explicit identification
\[
\C[G]\cong S\g_{\ol{1}}^*\otimes\C[G_0].
\]
In particular, we have a natural subalgebra $S\g_{-1}^*\otimes \C\cong S\g_{-1}^*\sub\C[G]$.  We claim that the adjoint vector field of $u$ preserves this subspace.  To show this we explain the form of the left and right infinitesimal translation of a vector $u\in\g$ on $\C[G]$:
\[
u_L(f)(X)(g)=-(-1)^{\ol{u}\ol{f}}f(\Ad(g)^{-1}(u)X)(g),
\]
\[
u_R(f)(X)(g)=(-1)^{\ol{u}(\ol{f}+\ol{X})}f(Xu)(g).
\]
Here $X\in\UU\g$ and $g\in G_0$.  Considering the left infinitesimal translation action first, we see that it will act by $0$ on $S\g_{-1}^*$ since $\Ad(g)(u)\in\g_{1}$ for all $g\in G_0$.  On the other hand it's not hard to see that the right infinitesimal translation by $u$ preserves $S\g_{-1}^*$ since $\g_{-1}$ is $\g_{0}$-stable.  Thus the adjoint action of $u$ induces a complex on $S\g_{-1}^*$.  

\begin{lemma}
	\begin{enumerate}
		\item Under the identification $\g_{-1}\cong\k$, $u$ induces a complex isomorphic to the cohomology differential on the Chevalley-Eilenberg complex $S\g_{-1}^*$.
		\item 
		\[
		DS_uS\g_{-1}^*\cong (S\g_{-1}^*)^{\k}\cong(\bigwedge\k^*)^{\k}.
		\]
		\item $S\g_{-1}^*$ is dual, as a complex over $u$, to $\UU\g/\UU(\g)(\g_{0}\oplus\g_{1})\cong\UU\g_{-1}$ under the canonical nondegenerate pairing $\UU\g\otimes\C[G]\to\C$.  Further this pairing descends to a nondegenerate pairing $(\UU\g_{-1})^{\k}\otimes(S\g_{-1}^*)^{\k}\to \C$.
		
		\item The map $(S\g_{-1}^*)^{\k}\to \C[G]_u$ is injective.

	\end{enumerate}
\end{lemma}
First we recall the Chevalley-Eilenberg complex on $S\g_{-1}^*$; given $f:S^k\g_{-1}\to\C$, the differential $d$ is given by
\[
df(x_1\cdots x_{k+1})=\sum\limits_{i<j}(-1)^{i+j}f([x_i,x_j],x_1\cdots\hat{x_i}\cdots\hat{x_j}\cdots,x_{k+1}).
\]
\begin{proof}
	For $f\in S\g_{-1}^*\sub\C[G]$, let $v_1,\dots,v_{k+1}\in\g_{-1}$. Write $V_j:=[v_j,u]\in\k$.  Then we have
	\begin{eqnarray*}
	(uf)(v_1\cdots v_{k+1})& = &(-1)^{\ol{f}+n}f(v_1\cdots v_{k+1}u)\\
	                   & = &(-1)^{\ol{f}}\sum\limits_{j}(-1)^{j-1}f(v_1\cdots[v_j,u]\cdots v_{k+1})\\
	                   & = &(-1)^{\ol{f}}\sum\limits_{i<j}(-1)^{j-1}f(v_1\cdots[v_i,V_j]\cdots\hat{v_j}\cdots v_{k+1})\\
	                   & = &(-1)^{\ol{f}}\sum\limits_{i<j}(-1)^{j+i}f([v_i,V_j]v_1\cdots\hat{v_i}\cdots\hat{v_j}\cdots v_{k+1})\\
	\end{eqnarray*}
Thus this differs from the Chevalley-Eilenberg complex only in the sign $(-1)^{\ol{f}}$, which is an ineffectual factor since the complex is graded.  This proves (1).  The proof of (2) follows from (1). 

The proof of (3) follows from the fact that the pairing $\UU\g\otimes\C[G]\to\C$ is nondegenerate and equivariant with respect to the adjoint actions on each factor.  Then (4) follows from (3).
\end{proof}

The above lemma tells us that we may view $(S\g_{-1}^*)^{\k}\cong(\bigwedge\k^*)^{\k}$ as the functions on a Lie supergroup and $(\UU\g_{-1})^{\k}$ as its enveloping superalgebra.

\begin{cor}
	Let $L_0$ be a $\g_{0}$-module which is semisimple over $\k$.  Then $DS_uK(L_0)$ is a free comodule over $(S\g_{-1}^*)^\k$ of rank $\dim L_0^{\k}$.
\end{cor}

\begin{proof}
	The structure of $DS_uK(L_0)$ as a comodule over $(S\g_{-1}^*)^\k$ comes from its structure over $\C[G]_u$.  On the other hand, we know that under the action of its enveloping algebra $(\UU\g_{-1})^{\k}$, $DS_uK(L_0)$ is a free module of rank $\dim L_0^{\k}$.  From this the statement follows.  
\end{proof}

We now note that the above results apply to the cases we have studied for $G=GL(n|n),P(n)$.  Further, for $GL(n|n)$ we further obtain:
\begin{cor}
	\begin{enumerate}
		\item Let $G=GL(n|n)$ and $u$ as before. Then $(S\g_{-1}^*)^{\k}\cong\C[G]_u$. 
	\end{enumerate}
\end{cor}
\begin{proof}
	We have shown that $(S\g_{-1}^*)^{\k}\to\C[G]_u$ is injective; on the other hand we saw in Example \ref{full rank gl computation example} that $\C[G]_u\cong H_{dR}^*(GL(n))$, and thus these superalgebras of the same dimension, meaning they are isomorphic.
\end{proof}

\textsc{\footnotesize Dept. of Mathematics, Ben Gurion University, Beer-Sheva,	Israel} 

\textit{\footnotesize Email address:} \texttt{\footnotesize xandersherm@gmail.com}

\end{document}